\documentclass[preprint]{acmart}

\usepackage{booktabs} 
\usepackage{favest}
\usepackage[ruled]{algorithm2e} 
\usepackage{multirow}
\usepackage[textfont=scriptsize,labelfont=scriptsize]{subcaption}

\SetAlFnt{\small}
\SetAlCapFnt{\small}
\SetAlCapNameFnt{\small}
\SetAlCapHSkip{0pt}
\IncMargin{-\parindent}

\allowdisplaybreaks

\acmJournal{TOMS}

\setcopyright{none}
\renewcommand\footnotetextcopyrightpermission[1]{} 
\usepackage{xpatch}
\makeatletter
\xpatchcmd{\ps@firstpagestyle}{Manuscript submitted to ACM}{}{\typeout{First patch succeeded}}{\typeout{first patch failed}}
\makeatother

\acmDOI{0000001.0000001}

\begin{document}
\title{Algorithm xxx: FaVeST --- Fast Vector Spherical Harmonic Transforms}

\author{Quoc T. Le Gia}
\affiliation{%
 \institution{The University of New South Wales}
 \city{Sydney}
 \state{NSW}
 \country{Australia}}
\email{qlegia@unsw.edu.au}
\author{Ming Li*}
\affiliation{%
  \institution{Zhejiang Normal University}
  \city{Jinhua}
 \state{Zhejiang}
 \country{China;}
  \institution{La Trobe University}
  \city{Melbourne}
 \state{VIC}
 \country{Australia}
}
\email{mingli@zjnu.edu.cn}
\author{Yu Guang Wang*}
\orcid{0000-0002-7450-0273}
\affiliation{%
 \institution{The University of New South Wales}
 \city{Sydney}
 \state{NSW}
 \country{Australia}}
\email{yuguang.wang@unsw.edu.au}
\thanks{*~ Corresponding authors}

\begin{abstract}
Vector spherical harmonics on the unit sphere of $\mathbb{R}^3$ have broad applications in geophysics, quantum mechanics and astrophysics. In the representation of a tangent vector field, one needs to evaluate the expansion and the Fourier coefficients of vector spherical harmonics. In this paper, we develop fast algorithms (FaVeST) for vector spherical harmonic transforms on these evaluations. The forward FaVeST evaluates the Fourier coefficients and has a computational cost proportional to $N\log \sqrt{N}$ for $N$ number of evaluation points. The adjoint FaVeST which evaluates a linear combination of vector spherical harmonics with a degree up to $\sqrt{M}$ for $M$ evaluation points has cost proportional to $M\log\sqrt{M}$. Numerical examples of simulated tangent fields illustrate the accuracy, efficiency and stability of FaVeST.
\end{abstract}

%


%
%



\maketitle

\renewcommand{\shortauthors}{Q.~T. Le Gia et al.}

\section{Introduction}
Vector spherical harmonics on the unit sphere $\sph{2}$ in $\Rd[3]$ are widely used in many areas such as astrophysics \cite{Planck2015Gravitational,Planck2018Gravitational,Planck2018I}, quantum mechanics \cite{Hill1954,Weinberg1994,Kim2009}, geophysics and geomagnetics \cite{BaEsGi1985,Freeden2009,Grafarend1986,LiBrOlWa2019,Lowes1966}, 3D fluid mechanics \cite{VeGuBiZo2009,VeRaBiZo2011}, global atmospherical modelling \cite{Holton1973,Giraldo2004,Fan_etal2018} and climate change modelling \cite{Swarztrauber1996,Swarztrauber2000,Swarztrauber2004}. For example, in constructing the numerical solution to
the Navier--Stokes equations on the unit sphere \cite{GaLeSl2011},
divergence-free vector spherical harmonics are used.
In simulating scattering waves by single or multiple spherical scatterers \cite{Nedelec2001, WaWaXi2018} modelled by the 3-dimensional Helmholtz equation, both divergence-free and curl-free vector spherical harmonics are used. In these problems,
the solutions which are vector fields are represented by an expansion of vector spherical harmonics. One then needs to evaluate the coefficients of vector spherical harmonics for a vector field and also a linear combination of vector spherical harmonics. They can be evaluated by the \emph{(discrete) forward vector spherical harmonic transform (FwdVSHT)} and \emph{adjoint vector spherical harmonic transform (AdjVSHT)} respectively.
In this paper, we develop fast algorithms for the forward and adjoint vector spherical harmonic transforms for tangent (vector) fields on $\sph{2}$ and their software implementation.

Let $\{(\dsh,\csh): \ell=1,2,\dots,\; m=-\ell,\dots,\ell\}$ be a set of pairs of the (complex-valued) divergence-free and curl-free vector spherical harmonics on $\sph{2}$. The coefficients for divergence-free and curl-free vector spherical harmonics are given by, for $\ell=1,2,\dots$, $m=-\ell,\dots,\ell$,
\begin{equation*}
    \dfco{T}=\int_{\sph{2}}T(\bx)\dsh^*(\bx)\ds\qquad \cfco{T}=\int_{\sph{2}}T(\bx)\csh^*(\bx)\ds,
\end{equation*}
where the $V^*$ is the complex conjugate transpose of the tangent field $V$.
The FwdVSHT for a spherical tangent field $T: \sph{2}\to\Cd[3]$ evaluates these coefficients by approximating the integrals of the coefficients with a quadrature rule which is a set of $N$ pairs of weights $w_i$ and points $\bx_i$ on $\sph{2}$:
\begin{equation}\label{eq:fwdvsht}
    \dfco{T}\approx\sum_{i=1}^{N}w_i T(\bx_i)\dsh^*(\bx_i)\qquad \cfco{T}\approx\sum_{i=1}^{N}w_i T(\bx_i)\csh^*(\bx_i).
\end{equation}
The AdjVSHT evaluates the expansion of $\dsh,\csh$ with two complex sequences $a_{\ell,m},b_{\ell,m}$, $\ell=1,2,\dots$, $m=-\ell,\dots,\ell$ as coefficients for a set of spherical points $\{\bx_i\}_{i=1}^M$, $M\geq1$:
\begin{equation}\label{eq:expan}
    \sum_{\ell=1}^{\infty}\sum_{m=-\ell}^{\ell}\left(a_{\ell,m}\dsh(\bx_i) + b_{\ell,m}\csh(\bx_i)\right),\quad i=1,\dots,M.
\end{equation}
To directly compute the summation of \eqref{eq:fwdvsht} for the coefficients with degree up to $L$ for $L\geq1$, the computational cost for FwdVSHT is $\bigo{NL^2}$. Here the quadrature rule should be properly chosen to minimize the approximation error. The optimal-order number of nodes for this purpose is $N=\bigo{L^2}$. (See Section~\ref{sec:complexity_error}.) Thus, the cost of direct computation for forward vector spherical harmonic transform is $\bigo{N^2}$. On the other hand, to directly compute the expansion \eqref{eq:expan} with truncation degree $L$ for $\ell$ incurs $\bigo{ML^2}$ computational steps. Then, to evaluate $M=\bigo{L^2}$ points, the computational complexity for direct evaluation of AdjVSHT is $\bigo{M^2}$.

In this paper, we develop fast computational strategies for FwdVSHT and AdjVSHT by explicit representations for the divergence-free and curl-free vector spherical harmonics in terms of scalar spherical harmonics. Here the close representation formula exploits Clebsch-Gordan coefficients in quantum mechanics. By this way, the fast scalar spherical harmonic transforms can be applied to speed up computation. The resulting algorithm reduces the computational cost to $\bigo{N\log\sqrt{N}}$ and $\bigo{M\log\sqrt{M}}$, both of which are nearly linear. We thus call the algorithms the \emph{\textbf{Fa}st \textbf{Ve}ctor \textbf{S}pherical Harmonic \textbf{T}ransform} or \emph{\fav}. A software package in Matlab is provided for implementing {\fav}.
We then validate the {\fav} algorithm by the numerical examples of simulated tangent fields on the sphere. The algorithm, accompanied by its software implementation fills the blank of fast algorithms for vector spherical harmonic transforms.

The rest of this paper is organized as follows. In Section~\ref{sec:related_works}, we review the fast Fourier transforms for scalar spherical harmonics on $\sph{2}$ and related works on the computation of vector spherical harmonic transforms. In Section~\ref{sec:vsh}, we introduce definitions and notation on scalar and vector spherical harmonics. In Sections~\ref{sec:FwdVSHT} and \ref{sec:AdjVSHT}, we give the representations for FwdVSHT and AdjVSHT in terms of forward and adjoint scalar spherical harmonic transforms. From these representations, in Section~\ref{sec:complexity_error}, we describe the fast algorithms ({\fav}) for the evaluation of FwdVSHT and AdjVSHT, and show that the computational complexity of the proposed {\fav} is nearly linear. We also estimate the approximation error for forward and adjoint {\fav}s. Section~\ref{sec:software} describes the software package in Matlab developed for {\fav}. In Section~\ref{sec:numer}, we give numerical examples of simulated spherical tangent fields to test {\fav}.
\section{Related Works}\label{sec:related_works}
Fast Fourier Transform (FFT) for $\Rd[d]$ is one of the most influential algorithms in science and engineering \cite{BrighamFFT,press2007numerical,tang2005dfti,pekurovsky2012p3dfft}. On the sphere, fast transforms for scalar spherical harmonics have been extensively studied by many researchers \cite{BaMaKl2018,Gorski_etal2005,KeKuPo2007,KeKuPo2009,healy2004towards,healy2003ffts,mohlenkamp1999fast,suda2004stability,suda2005fast,suda2002fast,Tygert2008,tygert2010fast,RoTy2006,KeKuPo2009,lu2013efficient,drake2008algorithm}.
 In particular, Keiner et al. provide the software library NFFT \cite{KeKuPo2009} which implements the fast forward and adjoint FFT algorithms for scalar spherical harmonics based on the non-equispaced FFT \cite{KeKuPo2007,kunis2003fast,keiner2008fast}. Their package is easy to use in Matlab environment and has been applied in many areas. Suda and Takami in \cite{suda2002fast} propose a fast scalar spherical harmonic transform algorithm with computational complexity $\mathcal{O}(N \log \sqrt{N})$ based on the divide-and-conquer approach with split Legendre functions (where $N$ is the number of nodes in the discretization for integral on $\sph{2}$), and the algorithm is used to solve the shallow water equation \cite{suda2005fast}. Rokhlin and Tygert \cite{RoTy2006} develop the fast algorithms for scalar spherical harmonic expansion with computational time proportional to $N\log \sqrt{N} \log(1/\epsilon)$ for a given precision $\epsilon>0$. Later, Tygert \cite{Tygert2008,tygert2010fast} improves their algorithm to achieve computational cost proportional to $N\log \sqrt{N}$ at any given precision. Reinecke and Seljebotn and Gorski et al. \cite{ReSe2013,Gorski_etal2005} develop the FFTs for spherical harmonics to evaluate Hierarchical Equal Area isoLatitude Pixelation (HEALPix) points \cite{Gorski_etal2005} in various programming languages. Their algorithm can work on millions of evaluation points with spherical harmonic degree up to $6,143$.

In contrast, fast transforms for vector spherical harmonics receive less attention. To the best of our knowledge, there are no existing fast algorithms for the forward and adjoint vector spherical harmonic transforms. Ganesh et al. \cite{GaLeSl2011} use FFTs to speed up their algorithms for solving Navier-Stokes PDEs on the unit sphere. However, their method is based on the idea that applies conventional FFTs to evaluate complex azimuthal exponential terms involved in the formulation of vector spherical harmonics. As fast Legendre transforms are not implemented, their method is not a fast transform. Wang et al. \cite{WaWaXi2018} evaluate vector spherical harmonic expansions via spectral element grids, which is neither fast computation.
\section{Vector Spherical Harmonics}\label{sec:vsh}
In this section, we present definitions and properties about spherical tangent fields, scalar and vector spherical harmonics and Clebsch-Gordan coefficients, see, e.g. \citep{DaXu2013, Edmonds2016},
which we will use in the representation of FwdVSHT and AdjVSHT in the next section.
A tangent (vector) field $T$ is a mapping from $\sph{2}$ to $\Cd[3]$ satisfying the normal component $(T\cdot\bx)\bx$ of $T$ is zero, here $T\cdot\bx:=\sum_{i=1}^3 T^{(i)}x^{(i)}$ is the inner product of $\Cd[3]$ for column vectors $T:=\bigl(T^{(1)},T^{(2)},T^{(3)}\bigr)'$ and $\bx:=(x^{(1)},x^{(2)},x^{(3)})'$, and the $'$ denotes the transpose of a vector (or matrix).
Let $\vLp{2}{2}$ be $L_2$ space of tangent fields on the sphere $\sph{2}$ with inner product
\begin{equation*}
    \InnerL{T,V} = \int_{\sph{2}}T^*(\bx) V(\bx)\mathrm{d}\sigma(\bx)
\end{equation*}
and $L_2$ norm $\|T\|_2=\sqrt{\InnerL{T,T}}$, where $T^*(\bx)$ is the complex conjugate transpose of $T(\bx)$.
Using spherical coordinates, the \emph{scalar spherical harmonics} can be explicitly written as, for $\ell=0,1,\dots$,
\begin{equation*}
\begin{aligned}
  \shY(\PT{x}) &:= \shY(\theta,\varphi) := \sqrt{\frac{2\ell+1}{4\pi}\frac{(\ell-m)!}{(\ell+m)!}}\aLegen{m}(\cos\theta)\: e^{\imu m\varphi}, \quad m=0,1,\ldots, \ell,\\
  \shY(\PT{x}) &:= (-1)^{m}\shY[\ell,-m](\PT{x}), \quad m=-\ell,\dots,-1.
 \end{aligned}
\end{equation*}
We would suppress the variable $\PT{x}$ in $\shY(\PT{x})$ if no confusion arises. 

In the following, we introduce vector spherical harmonics, see, e.g. \cite{Edmonds2016, Varshalovich_etal1988}.
For each $\ell=1,2,\dots$, $m=-\ell,\dots,\ell$, and integers $j_1,j_2,m_1,m_2$ satisfying $j_2\geq j_1\geq0$, $j_2-j_1\leq \ell\leq j_2+j_1$ and $-j_i\leq m_i\leq j_i$, $i=1,2$, the \emph{Clebsch-Gordan (CG) coefficients} are
\begin{equation*}
 C^{\ell,m}_{j_1,m_1,j_2,m_2} :=
(-1)^{(m+j_1-j_2)}\sqrt{2\ell+1}\left(\begin{array}{lll}j_1 & j_2 & \ell \\ m_1 &  m_2 & -m
\end{array} \right),
\end{equation*}
see e.g. \cite[Chapter~3.5]{MaPe2011}.
The \emph{covariant spherical basis vectors} are 
\begin{equation}\label{eq:cov.sph.basis.vec}
\mathbf{e}_{+1} = - \frac{1}{\sqrt{2}} \left([1,0,0] + i [0,1,0]\right)^T \qquad
  \mathbf{e}_{0} = [0,0,1]^T \qquad
  \mathbf{e}_{-1} = \frac{1}{\sqrt{2}}\left([1,0,0] - i [0,1,0]\right)^T.
\end{equation}
Let
\begin{equation}\label{eq:cl,dl,betal}
    c_{\ell}:=\sqrt{\frac{\ell+1}{2\ell+1}}\qquad d_{\ell}:=\sqrt{\frac{\ell}{2\ell+1}}.
\end{equation}
We define the coefficients
\begin{equation}\label{eq:Blm1}
\begin{aligned}
  B_{+1,\ell,m} &= c_{\ell} C^{\ell,m}_{\ell-1,m-1,1,1} \shY[\ell-1,m-1] + d_\ell C^{\ell,m}_{\ell+1,m-1,1,1} \shY[\ell+1,m-1]\\
  B_{0,\ell,m} &= c_\ell C^{\ell,m}_{\ell-1,m,1,0} \shY[\ell-1,m] + d_{\ell} C^{\ell,m}_{\ell+1,m,1,0} \shY[\ell+1,m]\\
  B_{-1,\ell,m} &= c_\ell C^{\ell,m}_{\ell-1,m+1,1,-1}\shY[\ell-1,m+1] + d_\ell C^{\ell,m}_{\ell+1,m+1,1,-1}\shY[\ell+1,m+1]
\end{aligned}
\end{equation}
and
\begin{equation}\label{eq:Blm2}
\begin{aligned}
  D_{+1,\ell,m} = i C^{\ell,m}_{\ell,m-1,1,1} \shY[\ell,m-1]\qquad
  D_{0,\ell,m} = i C^{\ell,m}_{\ell,m,1,0}\shY[\ell,m]
  \qquad
  D_{-1,\ell,m}= i C^{\ell,m}_{\ell,m+1,1,-1}\shY[\ell,m+1].
\end{aligned}
\end{equation}

\begin{definition}[Vector Spherical Harmonics]
    For $\ell=1,2,\dots$, $m=-\ell,\dots,\ell$, using the notation of \eqref{eq:cov.sph.basis.vec}, \eqref{eq:Blm1} and \eqref{eq:Blm2},
the \emph{divergence-free} and \emph{curl-free vector spherical harmonics} are defined by
\begin{equation}\label{eq:vsh}
\begin{aligned}
 \dsh = B_{+1,\ell,m} \mathbf{e}_{+1} +
                  B_{0,\ell,m} \mathbf{e}_{0} +
                  B_{-1,\ell,m} \mathbf{e}_{-1}\qquad
 \csh =  D_{+1,\ell,m} \mathbf{e}_{+1} +
                   D_{0,\ell,m} \mathbf{e}_{0}  +
                   D_{-1,\ell,m} \mathbf{e}_{-1}.
\end{aligned}
\end{equation}
Or equivalently, by \eqref{eq:cov.sph.basis.vec} and \eqref{eq:vsh},
\begin{align}\label{eq:vsh2}
    \dsh = \begin{pmatrix}
        -\frac{1}{\sqrt{2}} \left(B_{+1,\ell,m} - B_{-1,\ell,m}\right)\\
        -\frac{1}{\sqrt{2}}\imu \left(B_{+1,\ell,m} + B_{-1,\ell,m}\right)\\
        B_{0,\ell,m}
    \end{pmatrix}
    \qquad
    \csh = \begin{pmatrix}
        -\frac{1}{\sqrt{2}}\left(D_{+1,\ell,m}-D_{-1,\ell,m}\right)\\
        -\frac{1}{\sqrt{2}}\imu \left(D_{+1,\ell,m}+D_{-1,\ell,m}\right)\\
        D_{0,\ell,m}
    \end{pmatrix}.
\end{align}
\end{definition}
The set of vector spherical harmonics $\{\dsh,\csh: \ell=1,2,\dots, m=-\ell,\dots,\ell\}$ in \eqref{eq:vsh} or \eqref{eq:vsh2} forms an orthonormal basis for $\vLp{2}{2}$.
Using the property of $3$-$j$ symbols in \cite{NIST:DLMF}, the CG coefficients in \eqref{eq:Blm1} and \eqref{eq:Blm2} have the following explicit formula.
\begin{equation}\label{eq:ninecoeff}
\begin{array}{llll}
 & C^{\ell,m}_{\ell-1,m-1,1,1} = \displaystyle \sqrt{\frac{(\ell+m)(\ell+m-1)}{(2\ell)(2\ell-1)}}
 & &  C^{\ell,m}_{\ell+1,m-1,1,1} = \displaystyle \sqrt{\frac{(\ell-m+1)(\ell-m+2)}{(2\ell+2)(2\ell+3)}} \\[4mm]
 &\displaystyle C^{\ell,m}_{\ell-1,m,1,0} = \sqrt{\frac{(\ell+m)(\ell-m)}{\ell(2\ell-1)}} 
 & & \displaystyle C^{\ell,m}_{\ell+1,m,1,0} = -\sqrt{\frac{(\ell-m+1)(\ell+m+1)}{(2\ell+3)(\ell+1)}}  \\[4mm]
 & \displaystyle C^{\ell,m}_{\ell-1,m+1,1,-1} = \sqrt{\frac{(\ell-m)(\ell-m-1)}{(2\ell)(2\ell-1)}}
 & & \displaystyle C^{\ell,m}_{\ell+1,m+1,1,-1}= \sqrt{\frac{(\ell+m+1)(\ell+m+2)}{(2\ell+3)(2\ell+2)}}\\[4mm]
  &\displaystyle C^{\ell,m}_{\ell,m-1,1,1} = -\sqrt{\frac{(\ell+m)(\ell-m+1)}{\ell(2\ell+2)}} 
  & & \displaystyle C^{\ell,m}_{\ell,m+1,1,-1} = \sqrt{\frac{(\ell+m+1)(\ell-m)}{\ell(2\ell+2)}}\\[4mm]
   &\displaystyle C^{\ell,m}_{\ell,m,1,0} = \frac{m}{\sqrt{\ell(\ell+1)}}. &
\end{array}
\end{equation}
The expression of these CG coefficients in \eqref{eq:ninecoeff} will simplify computations in the proposed fast algorithms below, and also help with the interested readers follow the routines in our software.

\section{Fast Vector Spherical Harmonic Transforms}\label{sec:favest}
In Subsections~\ref{sec:FwdVSHT} and \ref{sec:AdjVSHT} below, we prove two theorems to represent the vector spherical harmonic transforms by scalar spherical harmonics and Clebsch-Gordan coefficients. From the representation formula, we obtain a computational strategy for fast evaluation of FwdVSHT and AdjVSHT.
Subsection~\ref{sec:complexity_error} shows the analysis of computational complexity and approximation error for the proposed {\fav}. In Subsection~\ref{sec:software}, we provide the ``user guide'' of the software implementation for {\fav} in Matlab environment.

\subsection{Fast Computation for FwdVSHT}\label{sec:FwdVSHT}
In this section, we provide an efficient way to evaluate the Fourier coefficients for vector spherical harmonics. The evaluation exploits the connection between the coefficients for vector and scalar spherical harmonics. This connection together with FFTs for scalar spherical harmonics allows fast computation of FwdVSHT and AdjVSHT.

The \emph{divergence-free and curl-free coefficients} of a tangent field $T$ on $\sph{2}$ are, for $\ell=1,2,\dots$, $m=-\ell,\dots,\ell$,
\begin{equation*}
    \dfco{T} := \InnerL{T,\dsh}\qquad \cfco{T} := \InnerL{T,\csh}.
\end{equation*}
To numerically evaluate them, one needs to discretize the integrals of the coefficients by a \emph{quadrature rule} \cite{HeSlWo2015}, which is a set $\QN:=\{(w_i,\bx_i)\}_{i=1}^{N}$ of $N$, $N\geq2$, pairs of real numbers and points on $\sph{2}$.
\begin{definition}[FwdVSHT]\label{dfn:fwdvsht}
    For a sequence of column vectors $\{T_{k}\}_{k=1}^{N}$ in $\Rd[3]$, \emph{discrete forward divergence-free and curl-free transforms} for $\{T_k\}_{k=1}^{N}$ associated with $\QN$, or simply \emph{FwdVSHT}, are the weighted sums
\begin{equation}\label{eq:fwdvsht2}
\begin{aligned}
    \dflm(T_{\cdot}) := \dflm(\{T_k\}_{k=1}^{N},\QN):= \sum_{k=1}^{N}w_k \dsh^{*}(\bx_k)T_k\\[1mm]
    \cflm(T_{\cdot}) := \cflm(\{T_k\}_{k=1}^{N},\QN):= \sum_{k=1}^{N}w_k \csh^{*}(\bx_k)T_k.
\end{aligned}
\end{equation}
\end{definition}
For a tangent field $T$, the divergence-free and curl-free coefficients can be approximated by FwdVSHT for the sequence of values $\{T(\bx_i)\}_{i=1}^N$ of the tangent field at quadrature nodes $\{\bx_i\}_{i=1}^N$:
\begin{equation}\label{eq:dfco_cfco}
\begin{aligned}
    \dfco{T}\approx \dflm(\{T(\bx_k)\}_{k=1}^{N},\QN):= \sum_{k=1}^{N}w_k \dsh^{*}(\bx_k)T(\bx_k)\\[1mm]
    \cfco{T}\approx \cflm(\{T(\bx_k)\}_{k=1}^{N},\QN):= \sum_{k=1}^{N}w_k \csh^{*}(\bx_k)T(\bx_k).
\end{aligned}
\end{equation}
As mentioned, we call \eqref{eq:fwdvsht2} \emph{forward vector spherical harmonic transform}, or \emph{FwdVSHT} for tangent field $T$.

Now, we design fast computation for FwdVSHT in Definition~\ref{dfn:fwdvsht} using the scalar version of FwdVSHT. Let $\{f_k\}_{k=1}^{N}$ be a real sequence. For $\ell\geq1, m=-\ell,\dots,\ell$, the \emph{discrete forward scalar spherical harmonic transform (FwdSHT)} is
\begin{equation}\label{eq:Flm}
    \flm(f_{\cdot}):=\flm(f_{k}):= \flm(\{f_k\}_{k=1}^{N},\QN) := \sum_{k=1}^{N}w_k f_k\shY^{*}(\bx_k).
\end{equation}
They are approximations of spherical harmonic coefficients $
\langle f,\shY\rangle$ for $f$ by the quadrature rule $\QN$. 
The following theorem shows a representation of the FwdVSHT by FwdSHT for a sequence $\{T_k\}_{k=1}^{N}$, which would allow us to efficiently compute $\dflm(T_{\cdot})$ and $\cflm(T_{\cdot})$.
For each $k=1,\dots,N$, let $\bigl(T^{(1)}_k,T^{(2)}_k,T^{(3)}_k\bigr)$ be the components of the vector $T_k$. For $\ell=1,\dots$, $m=-\ell,\dots,\ell$, using the notation of \eqref{eq:cl,dl,betal} and \eqref{eq:ninecoeff}, we define the coefficients
\begin{equation}\label{eq:xi_klm}
	\begin{array}{lll}
		&\xi^{(1)}_{\ell,m}:= c_{\ell+1} C_{\ell,m,1,1}^{\ell+1,m+1}\qquad
		&\xi^{(2)}_{\ell,m}:= d_{\ell-1} C_{\ell,m,1,1}^{\ell-1,m+1}\\[2mm]
		&\xi^{(3)}_{\ell,m}:= c_{\ell+1} C_{\ell,m,1,-1}^{\ell+1,m-1}\qquad
		&\xi^{(4)}_{\ell,m}:= d_{\ell-1} C_{\ell,m,1,-1}^{\ell-1,m-1}\\[2mm]
		&\xi^{(5)}_{\ell,m}:= c_{\ell+1} C_{\ell,m,1,0}^{\ell+1,m}\qquad
		&\xi^{(6)}_{\ell,m}:= d_{\ell-1} C_{\ell,m,1,0}^{\ell-1,m}
	\end{array}
\end{equation}
and
\begin{equation}\label{eq:mu_klm}
		\mu^{(1)}_{\ell,m} := C_{\ell,m,1,1}^{\ell,m+1}\qquad
		\mu^{(2)}_{\ell,m} := C_{\ell,m,1,0}^{\ell,m}\qquad
		\mu^{(3)}_{\ell,m} := C_{\ell,m,1,-1}^{\ell,m-1}.
\end{equation}

\begin{theorem}\label{thm:almblmviaFlm} Let $\{T_{k}\}_{k=1}^{N}$ be a sequence of column vectors in $\Rd[3]$ and $\QN:=\{(w_i,\bx_i)\}_{i=1}^{N}$ a quadrature rule on $\sph{2}$. Then, for $\ell=1,2,\dots$, $m=-\ell,\dots,\ell$, the FwdVSHT can be represented by its scalar version FwdSHT, as follows.
\begin{equation}\label{eq:almblmviaFlm}
\begin{aligned}
		\dflm(T_{\cdot})
				& = \frac{1}{\sqrt{2}}\biggl\{\xi_{\ell-1,m-1}^{(1)}\left[-\flm[\ell-1,m-1]\left(T^{(1)}_{\cdot}\right)+\imu\: \flm[\ell-1,m-1]\left(T^{(2)}_{\cdot}\right)\right]
		  + \xi_{\ell+1,m-1}^{(2)}\left[-\flm[\ell+1,m-1]\left(T^{(1)}_{\cdot}\right)+\imu\: \flm[\ell+1,m-1]\left(T^{(2)}_{\cdot}\right)\right] \\
			&\qquad+ \xi_{\ell-1,m+1}^{(3)}\left[\flm[\ell-1,m+1]\left(T^{(1)}_{\cdot}\right)+\imu \:\flm[\ell-1,m+1]\left(T^{(2)}_{\cdot}\right)\right]+
			\xi_{\ell+1,m+1}^{(4)}\left[\flm[\ell+1,m+1]\left(T^{(1)}_{\cdot}\right)+\imu\:\flm[\ell+1,m+1]\left(T^{(2)}_{\cdot}\right)\right]\bigg\}\\
			&\qquad +
			\xi_{\ell-1,m}^{(5)}\flm[\ell-1,m]\left(T^{(3)}_{\cdot}\right) +
			\xi_{\ell+1,m}^{(6)}\flm[\ell+1,m]\left(T^{(3)}_{\cdot}\right)\\[2mm]
		\cflm(T_{\cdot})
	&= -\frac{1}{\sqrt{2}}\imu\left[\mu_{\ell,m-1}^{(1)}\left(-\flm[\ell,m-1]\bigl(T_{\cdot}^{(1)}\bigr)+\imu\: \flm[\ell,m-1]\bigl(T_{\cdot}^{(2)}\bigr)\right) +  \mu_{\ell,m+1}^{(3)}\left(\flm[\ell,m+1]\bigl(T_{\cdot}^{(1)}\bigr)+\imu\: \flm[\ell,m+1]\bigl(T_{\cdot}^{(2)}\bigr)\right)\right]\\
	&\qquad-\imu \mu_{\ell,m}^{(2)} \flm\bigl(T_{\cdot}^{(3)}\bigr),
	\end{aligned}
	\end{equation}
	where we use the notation of \eqref{eq:fwdvsht2}, \eqref{eq:Flm}, \eqref{eq:xi_klm} and \eqref{eq:mu_klm}.
\end{theorem}

\begin{proof} Let $\ell=1,2,\dots,$ and $m=-\ell,\dots,\ell$. By \eqref{eq:Blm1} and \eqref{eq:xi_klm}, the discrete forward divergence-free transform
	\begin{align*}
		\dflm(T_{\cdot}) &= \sum_{k=1}^{N}w_k \dsh^{*}(\bx_k)T_k\\
		&= \sum_{k=1}^{N} w_k \biggl(-\frac{1}{\sqrt{2}}T^{(1)}_k \bigl(B_{+1,\ell,m}^{*}(\bx_k)-B^{*}_{-1,\ell,m}(\bx_k)\bigr)
		+\frac{1}{\sqrt{2}} \imu T^{(2)}_k \bigl(B^{*}_{+1,\ell,m}(\bx_k)+B^{*}_{-1,\ell,m}(\bx_k)\bigr)
		+ T^{(3)}_{k}B^{*}_{0,\ell,m}(\bx_k)
		\biggr)\\
		&= \sum_{k=1}^{N} w_k \biggl[
		\Bigl(\frac{1}{\sqrt{2}}\bigl(-T^{(1)}_{k}+\imu T^{(2)}_{k}\bigr)c_{\ell} C_{\ell-1,m-1,1,1}^{\ell,m}\Bigr)\shY[\ell-1,m-1]^{*}(\bx_k)\\
		&\qquad + \Bigl(\frac{1}{\sqrt{2}}\bigl(-T^{(1)}_{k}+\imu T^{(2)}_{k}\bigr)d_{\ell} C_{\ell+1,m-1,1,1}^{\ell,m}\Bigr)\shY[\ell+1,m-1]^{*}(\bx_k)\\
		&\qquad + \Bigl(\frac{1}{\sqrt{2}}\bigl(T^{(1)}_{k}+\imu T^{(2)}_{k}\bigr)c_{\ell} C_{\ell-1,m+1,1,-1}^{\ell,m}\Bigr)\shY[\ell-1,m+1]^{*}(\bx_k)\\
		&\qquad + \Bigl(\frac{1}{\sqrt{2}}\bigl(T^{(1)}_{k}+\imu T^{(2)}_{k}\bigr)d_{\ell} C_{\ell+1,m+1,1,-1}^{\ell,m}\Bigr)\shY[\ell+1,m+1]^{*}(\bx_k)\\
		&\qquad + \bigl(T^{(3)}_{k} c_{\ell} C_{\ell-1,m,1,0}^{\ell,m}\bigr) \shY[\ell-1,m]^{*}(\bx_k)\\
		&\qquad + \bigl(T^{(3)}_{k} d_{\ell} C_{\ell+1,m,1,0}^{\ell,m}\bigr) \shY[\ell+1,m]^{*}(\bx_k)
		\biggr]\\
		& = \frac{1}{\sqrt{2}}\biggl\{\xi_{\ell-1,m-1}^{(1)}\left[-\flm[\ell-1,m-1]\left(T^{(1)}_{\cdot}\right)+\imu\: \flm[\ell-1,m-1]\left(T^{(2)}_{\cdot}\right)\right]
		  + \xi_{\ell+1,m-1}^{(2)}\left[-\flm[\ell+1,m-1]\left(T^{(1)}_{\cdot}\right)+\imu\: \flm[\ell+1,m-1]\left(T^{(2)}_{\cdot}\right)\right] \\
			&\qquad+ \xi_{\ell-1,m+1}^{(3)}\left[\flm[\ell-1,m+1]\left(T^{(1)}_{\cdot}\right)+\imu \:\flm[\ell-1,m+1]\left(T^{(2)}_{\cdot}\right)\right]+
			\xi_{\ell+1,m+1}^{(4)}\left[\flm[\ell+1,m+1]\left(T^{(1)}_{\cdot}\right)+\imu\:\flm[\ell+1,m+1]\left(T^{(2)}_{\cdot}\right)\right]\bigg\}\\
			&\qquad +
			\xi_{\ell-1,m}^{(5)}\flm[\ell-1,m]\left(T^{(3)}_{\cdot}\right) +
			\xi_{\ell+1,m}^{(6)}\flm[\ell+1,m]\left(T^{(3)}_{\cdot}\right).
	\end{align*}
In a similar way, for the curl-free case, we use \eqref{eq:Blm2} and \eqref{eq:mu_klm} to obtain
\begin{align*}
	\cflm(T_{\cdot}) &:= \sum_{k=1}^{N}w_k \csh^{*}(\bx_k)T(\bx_k)\\
	& = \sum_{k=1}^{N}w_k \left\{-\frac{1}{\sqrt{2}}\imu\left[\mu_{\ell,m-1}^{(1)}\shY[\ell,m-1]^{*}(\bx_k)\left(-T^{(1)}_{k}+\imu T^{(2)}_k\right)
	+ \mu_{\ell,m+1}^{(3)}\shY[\ell,m+1]^{*}(\bx_k)\left(T^{(1)}_{k}+\imu T^{(2)}_k\right)\right]\right.\\
	&\qquad \left. - \imu \mu_{\ell,m}^{(2)}\shY[\ell,m]^{*}(\bx_k)T^{(3)}_{k}\right\}\\
	&= -\frac{1}{\sqrt{2}}\imu\left[\mu_{\ell,m-1}^{(1)}\left(-\flm[\ell,m-1]\bigl(T_{\cdot}^{(1)}\bigr)+\imu\: \flm[\ell,m-1]\bigl(T_{\cdot}^{(2)}\bigr)\right) +  \mu_{\ell,m+1}^{(3)}\left(\flm[\ell,m+1]\bigl(T_{\cdot}^{(1)}\bigr)+\imu\: \flm[\ell,m+1]\bigl(T_{\cdot}^{(2)}\bigr)\right)\right]\\
	&\qquad-\imu \mu_{\ell,m}^{(2)} \flm\bigl(T_{\cdot}^{(3)}\bigr),
\end{align*}
thus completing the proof.
\end{proof}

\subsection{Fast Computation for AdjVSHT}\label{sec:AdjVSHT}
In this section, we study the adjoint vector spherical harmonic transforms for a tangent field on the sphere.
\begin{definition}[AdjVSHT]
	For $L\geq1$ and two complex sequences $\{a_{\ell,m},b_{\ell,m}: \ell=1,2,\dots,\; m=-\ell,\dots,\ell\}$, the \emph{adjoint vector spherical harmonic transform} or \emph{AdjVSHT} of degree $L$ is the Fourier partial sum  
\begin{equation}\label{eq:SLvec}
	\bbS_{L}({a_{\ell,m}},{b_{\ell,m}};\bx)
	:=\sum_{\ell=1}^{L}\sum_{m=-\ell}^{\ell} \left(a_{\ell,m} \dsh(\bx) + b_{\ell,m} \csh(\bx)\right),\quad \bx\in\sph{2}.
\end{equation}
Or equivalently, by \eqref{eq:cov.sph.basis.vec} and \eqref{eq:vsh},
\begin{equation}\label{eq:SLvecviaBlm.Dlm}
	\bbS_{L}({a_{\ell,m}},{b_{\ell,m}})
	= \sum_{\ell=1}^{L}\sum_{m=-\ell}^{\ell}\begin{pmatrix}
		-\frac{1}{\sqrt{2}} \left(a_{\ell,m}B_{+1,\ell,m} - a_{\ell,m}B_{-1,\ell,m}\right)\\
		-\frac{1}{\sqrt{2}}\imu \left(a_{\ell,m}B_{+1,\ell,m} + a_{\ell,m}B_{-1,\ell,m}\right)\\
		a_{\ell,m} B_{0,\ell,m}
	\end{pmatrix}
	+ \sum_{\ell=1}^{L}\sum_{m=-\ell}^{\ell}\begin{pmatrix}
		-\frac{1}{\sqrt{2}}\left(b_{\ell,m}D_{+1,\ell,m}-b_{\ell,m}D_{-1,\ell,m}\right)\\
		-\frac{1}{\sqrt{2}}\imu \left(b_{\ell,m}D_{+1,\ell,m}+b_{\ell,m}D_{-1,\ell,m}\right)\\
		b_{\ell,m}D_{0,\ell,m}.
	\end{pmatrix}.
\end{equation}
\end{definition}

For $L\ge0$ and a finite complex sequence $g_{\ell,m}$, $\ell=0,1,\dots,L$, $m=-\ell,\dots,\ell$, the \emph{adjoint scalar spherical harmonic transform} or \emph{AdjSHT} of degree $L$ is the Fourier partial sum of scalar spherical harmonics $\shY$:
\begin{equation}\label{eq:SLscalar}
	S_{L}(g_{\ell,m}) := \sum_{\ell=0}^{L}\sum_{m=-\ell}^{\ell}g_{\ell,m} \shY.
\end{equation}

The AdjVSHT can be represented by its scalar version with CG coefficients, as shown by the following theorem.
With the notation of \eqref{eq:cl,dl,betal} and \eqref{eq:ninecoeff}, we define the coefficients
\begin{equation}\label{eq:nulm1}
\begin{aligned}
	\nu_{\ell,m}^{(1)} &:= c_{\ell+1}\left(a_{\ell+1,m+1}C_{\ell,m,1,1}^{\ell+1,m+1}
	-a_{\ell+1,m-1}C_{\ell,m,1,-1}^{\ell+1,m-1}\right),\quad \ell=0,\dots,L-1,\; m=-\ell,\dots,\ell,\\[1mm]
	\nu_{\ell,m}^{(2)} &:= \left\{\begin{array}{ll}d_{\ell-1}\left(a_{\ell-1,m+1}C_{\ell,m,1,1}^{\ell-1,m+1}
	-a_{\ell-1,m-1}C_{\ell,m,1,-1}^{\ell-1,m-1}\right), & \ell=2,\dots,L+1,\; |m|=0,1,\dots,\ell-2,\\
	0, & \ell=0, 1 \mbox{~or~} |m|=\ell-1,\; \ell,
	\end{array}\right.\\
	\nu_{\ell,m}^{(3)} &:= i c_{\ell+1}\left(a_{\ell+1,m+1}C_{\ell,m,1,1}^{\ell+1,m+1}
	+ a_{\ell+1,m-1}C_{\ell,m,1,-1}^{\ell+1,m-1}\right),\quad \ell=0,\dots,L-1,\; m=-\ell,\dots,\ell,\\
	\nu_{\ell,m}^{(4)} &:= \left\{\begin{array}{ll}
i d_{\ell-1}\left(a_{\ell-1,m+1}C_{\ell,m,1,1}^{\ell-1,m+1}
	+ a_{\ell-1,m-1}C_{\ell,m,1,-1}^{\ell-1,m-1}\right), & \ell=2,\dots,L+1,\; m=-\ell,\dots,\ell,\\
	0, & \ell=0,1 \mbox{~or~} |m| = \ell-2,\;\ell-1,
	\end{array}\right.
\end{aligned}
\end{equation}
\begin{equation}\label{eq:nulm2}
\begin{aligned}
	\nu_{\ell,m}^{(5)} &:= a_{\ell+1,m}c_{\ell+1} C_{\ell,m,1,0}^{\ell+1,m},\quad \ell=0,\dots,L-1,\; m=-\ell,\dots,\ell,\\
	\nu_{\ell,m}^{(6)} &:= \left\{
	\begin{array}{ll}
		a_{\ell-1,m}d_{\ell-1}C_{\ell,m,1,0}^{\ell-1,m} , & \ell=2,\dots,L+1,\; |m|=0,1,\dots,\ell-1,\\
		0,& \ell=0,1 \mbox{~or~} |m|=\ell,
	\end{array}
	\right.
\end{aligned}
\end{equation}
and
\begin{equation}\label{eq:etalm}
\begin{aligned}
		\eta_{\ell,m}^{(1)} &:= \left\{\begin{array}{ll}\imu\left(b_{\ell,m+1}C_{\ell,m,1,1}^{\ell,m+1}
	-b_{\ell,m-1} C_{\ell,m,1,-1}^{\ell,m-1}\right), & \ell=1,\dots,L, \; m=-\ell,\dots,\ell,\\
	0, & \ell=0,\; m = 0,
	\end{array}\right.\\
		\eta_{\ell,m}^{(2)} &:= \left\{\begin{array}{ll}	
b_{\ell,m+1}C_{\ell,m,1,1}^{\ell,m+1}
	+b_{\ell,m-1} C_{\ell,m,1,-1}^{\ell,m-1}, & \ell=1,\dots,L,\; m=-\ell,\dots,\ell,\\
	0, & \ell=0,\; m=0,
	\end{array}\right.\\
	\eta_{\ell,m}^{(3)}&:= \left\{\begin{array}{ll}
 \imu b_{\ell,m}C_{\ell,m,1,0}^{\ell,m}, & \ell=1,\dots,L,\; m=-\ell,\dots,\ell,\\
	0, & \ell=0,\; m=0.
	\end{array}\right.
\end{aligned}
\end{equation}

\begin{theorem}\label{thm:SLvecviascal}
Let $\{a_{\ell,m},b_{\ell,m}: \ell=1,2,\dots,\; m=-\ell,\dots,\ell\}$ be two complex sequences. For $L\geq1$, the AdjVSHT for $a_{\ell,m}, b_{\ell,m}$ can be represented by its scalar version AdjSHT, as follows.
	\begin{equation}\label{eq:SLvecviascal}
	\bbS_{L}(a_{\ell,m},b_{\ell,m})
	= \begin{pmatrix}
		-\frac{1}{\sqrt{2}} \left(S_{L-1}\bigl(\nu^{(1)}_{\ell,m}\bigr) + S_{L+1}\bigl(\nu^{(2)}_{\ell,m}\bigr)+S_{L}\bigl(\eta_{\ell,m}^{(1)}\bigr)\right)\\[2mm]
		-\frac{1}{\sqrt{2}} \left(S_{L-1}\bigl(\nu^{(3)}_{\ell,m}\bigr) + S_{L+1}\bigl(\nu^{(4)}_{\ell,m}\bigr) - S_{L}\bigl(\eta_{\ell,m}^{(2)}\bigr)\right)\\[2mm]
		S_{L-1}(\nu_{\ell,m}^{(5)}) + S_{L+1}(\nu_{\ell,m}^{(6)}) + S_{L}(\eta_{\ell,m}^{(3)})
		\end{pmatrix},
\end{equation}
where we use the notation of \eqref{eq:SLscalar}, \eqref{eq:nulm1}, \eqref{eq:nulm2} and \eqref{eq:etalm}.
\end{theorem}
\begin{proof}
By \eqref{eq:Blm1}, we write each component of \eqref{eq:SLvecviaBlm.Dlm} by AdjSHT, as follows. For the divergence-free term in \eqref{eq:SLvecviaBlm.Dlm},
\begin{align*}
	&-\frac{1}{\sqrt{2}}\sum_{\ell=1}^{L}\sum_{m=-\ell}^{\ell} \left(a_{\ell,m}B_{+1,\ell,m} - a_{\ell,m}B_{-1,\ell,m}\right)\\
	=& -\frac{1}{\sqrt{2}}\biggl[\sum_{\ell=1}^{L}\sum_{m=-\ell}^{\ell}a_{\ell,m}\left(c_{\ell}C_{\ell-1,m-1,1,1}^{\ell,m}\right)\shY[\ell-1,m-1]
	+ \sum_{\ell=1}^{L}\sum_{m=-\ell}^{\ell}a_{\ell,m}\left(d_{\ell}C_{\ell+1,m-1,1,1}^{\ell,m}\right)\shY[\ell+1,m-1]\\
	&\quad\qquad - \sum_{\ell=1}^{L}\sum_{m=-\ell}^{\ell}a_{\ell,m}\left(c_{\ell}C_{\ell-1,m+1,1,-1}^{\ell,m}\right)\shY[\ell-1,m+1]
	- \sum_{\ell=1}^{L}\sum_{m=-\ell}^{\ell}a_{\ell,m}\left(d_{\ell}C_{\ell+1,m+1,1,-1}^{\ell,m}\right)\shY[\ell+1,m+1]
	\biggr]\\
	=& -\frac{1}{\sqrt{2}}\biggl[\sum_{\ell=0}^{L-1}\sum_{m=-\ell-2}^{\ell}a_{\ell+1,m+1}\left(c_{\ell+1}C_{\ell,m,1,1}^{\ell+1,m+1}\right)\shY
	+ \sum_{\ell=2}^{L+1}\sum_{m=-\ell}^{\ell-2}a_{\ell-1,m+1}\left(d_{\ell-1}C_{\ell,m,1,1}^{\ell-1,m+1}\right)\shY\\
	&\quad\qquad - \sum_{\ell=0}^{L-1}\sum_{m=-\ell}^{\ell+2}a_{\ell+1,m-1}\left(c_{\ell+1}C_{\ell,m,1,-1}^{\ell+1,m-1}\right)\shY
	- \sum_{\ell=2}^{L+1}\sum_{m=-\ell+2}^{\ell}a_{\ell-1,m-1}\left(d_{\ell-1}C_{\ell,m,1,-1}^{\ell-1,m-1}\right)\shY
	\biggr].
\end{align*}
This and \eqref{eq:ninecoeff} give
\begin{align}
	&-\frac{1}{\sqrt{2}}\sum_{\ell=1}^{L}\sum_{m=-\ell}^{\ell} \left(a_{\ell,m}B_{+1,\ell,m} - a_{\ell,m}B_{-1,\ell,m}\right)\notag\\
	&\qquad = -\frac{1}{\sqrt{2}}\left[\sum_{\ell=0}^{L-1}\sum_{m=-\ell-2}^{\ell+2}c_{\ell+1}\left(a_{\ell+1,m+1}C_{\ell,m,1,1}^{\ell+1,m+1}
	-a_{\ell+1,m-1}C_{\ell,m,1,-1}^{\ell+1,m-1}\right)\shY\right.\notag\\
	&\qquad\qquad \left.+ \sum_{\ell=2}^{L+1}\sum_{m=-\ell+2}^{\ell-2}d_{\ell-1}\left(a_{\ell-1,m+1}C_{\ell,m,1,1}^{\ell-1,m+1}
	-a_{\ell-1,m-1}C_{\ell,m,1,-1}^{\ell-1,m-1}\right)\shY\right],\label{eq:first.component1}
\end{align}
where for the $|m|>\ell$ (which exceeds the range of $m$ for spherical harmonics), $\shY=0$.
We then let
\begin{align*}
	\nu_{\ell,m}^{(1)} &:= c_{\ell+1}\left(a_{\ell+1,m+1}C_{\ell,m,1,1}^{\ell+1,m+1}
	-a_{\ell+1,m-1}C_{\ell,m,1,-1}^{\ell+1,m-1}\right),\quad \ell=0,\dots,L-1,\; m=-\ell,\dots,\ell,\\[1mm]
	\nu_{\ell,m}^{(2)} &:= \left\{\begin{array}{ll}d_{\ell-1}\left(a_{\ell-1,m+1}C_{\ell,m,1,1}^{\ell-1,m+1}
	-a_{\ell-1,m-1}C_{\ell,m,1,-1}^{\ell-1,m-1}\right), & \ell=2,\dots,L+1,\; |m|=0,1,\dots,\ell-2,\\
	0, & \ell=0, 1 \mbox{~or~} |m|=\ell-1,\; \ell,
	\end{array}\right.
\end{align*}
by which and \eqref{eq:first.component1},

\begin{equation*}
	-\frac{1}{\sqrt{2}}\sum_{\ell=1}^{L}\sum_{m=-\ell}^{\ell} \left(a_{\ell,m}B_{+1,\ell,m} - a_{\ell,m}B_{-1,\ell,m}\right) =
	-\frac{1}{\sqrt{2}} \left(S_{L-1}\bigl(\nu^{(1)}_{\ell,m}\bigr) + S_{L+1}\bigl(\nu^{(2)}_{\ell,m}\bigr) \right).
\end{equation*}
Similarly,
\begin{align*}
	&-\frac{1}{\sqrt{2}}\imu\sum_{\ell=1}^{L}\sum_{m=-\ell}^{\ell} \left(a_{\ell,m}B_{+1,\ell,m} + a_{\ell,m}B_{-1,\ell,m}\right)\notag\\
	&\qquad = -\frac{1}{\sqrt{2}}\left[\sum_{\ell=0}^{L-1}\sum_{m=-\ell-2}^{\ell+2}\imu c_{\ell+1}\left(a_{\ell+1,m+1}C_{\ell,m,1,1}^{\ell+1,m+1}
	+ a_{\ell+1,m-1}C_{\ell,m,1,-1}^{\ell+1,m-1}\right)\shY\right.\notag\\
	&\qquad\qquad \left.+ \sum_{\ell=2}^{L+1}\sum_{m=-\ell+2}^{\ell-2}\imu d_{\ell-1}\left(a_{\ell-1,m+1}C_{\ell,m,1,1}^{\ell-1,m+1}
	+ a_{\ell-1,m-1}C_{\ell,m,1,-1}^{\ell-1,m-1}\right)\shY\right].\label{eq:first.component1}
\end{align*}
Let
\begin{align*}
	\nu_{\ell,m}^{(3)} &:= i c_{\ell+1}\left(a_{\ell+1,m+1}C_{\ell,m,1,1}^{\ell+1,m+1}
	+ a_{\ell+1,m-1}C_{\ell,m,1,-1}^{\ell+1,m-1}\right),\quad \ell=0,\dots,L-1,\; m=-\ell,\dots,\ell,\\[1mm]
	\nu_{\ell,m}^{(4)} &:= \left\{\begin{array}{ll}
i d_{\ell-1}\left(a_{\ell-1,m+1}C_{\ell,m,1,1}^{\ell-1,m+1}
	+ a_{\ell-1,m-1}C_{\ell,m,1,-1}^{\ell-1,m-1}\right), & \ell=2,\dots,L+1,\; m=-\ell,\dots,\ell,\\
	0, & \ell=0,1 \mbox{~or~} |m| = \ell-2,\;\ell-1,
	\end{array}\right.
\end{align*}
then
\begin{equation*}
	-\frac{1}{\sqrt{2}}\imu\sum_{\ell=1}^{L}\sum_{m=-\ell}^{\ell} \left(a_{\ell,m}B_{+1,\ell,m} + a_{\ell,m}B_{-1,\ell,m}\right) =
	-\frac{1}{\sqrt{2}} \left(S_{L-1}\bigl(\nu^{(3)}_{\ell,m}\bigr) + S_{L+1}\bigl(\nu^{(4)}_{\ell,m}\bigr)\right).
\end{equation*}
As
\begin{equation*}
	\sum_{\ell=1}^{L}\sum_{m=-\ell}^{\ell}a_{\ell,m} B_{0,\ell,m}
	= \sum_{\ell=0}^{L-1}\sum_{m=-\ell}^{\ell}a_{\ell+1,m}c_{\ell+1} C_{\ell,m,1,0}^{\ell+1,m} \shY +
	\sum_{\ell=2}^{L+1}\sum_{m=-\ell+1}^{\ell-1} a_{\ell-1,m}d_{\ell-1}C_{\ell,m,1,0}^{\ell-1,m} \shY,
\end{equation*}
we let
\begin{align*}
	\nu_{\ell,m}^{(5)} &:= a_{\ell+1,m}c_{\ell+1} C_{\ell,m,1,0}^{\ell+1,m},\quad \ell=0,\dots,L-1,\; m=-\ell,\dots,\ell,\\
	\nu_{\ell,m}^{(6)} &:= \left\{
	\begin{array}{ll}
		a_{\ell-1,m}d_{\ell-1}C_{\ell,m,1,0}^{\ell-1,m} , & \ell=2,\dots,L+1,\; |m|=0,1,\dots,\ell-1,\\
		0,& \ell=0,1 \mbox{~or~} |m|=\ell,
	\end{array}
	\right.
\end{align*}
then,
\begin{equation*}
	\sum_{\ell=1}^{L}\sum_{m=-\ell}^{\ell}a_{\ell,m} B_{0,\ell,m}
	= S_{L-1}(\nu_{\ell,m}^{(5)}) + S_{L+1}(\nu_{\ell,m}^{(6)}).
\end{equation*}

For the curl-free term in \eqref{eq:SLvecviaBlm.Dlm},
\begin{align*}
	&-\frac{1}{\sqrt{2}}\sum_{\ell=1}^{L}\sum_{m=-\ell}^{\ell}
	\left(b_{\ell,m}D_{+1,\ell,m}-b_{\ell,m}D_{-1,\ell,m}\right)\notag\\[1mm]
	&\qquad= -\frac{1}{\sqrt{2}}
	\left(\sum_{\ell=1}^{L}\sum_{m=-\ell}^{\ell}\imu b_{\ell,m}C_{\ell,m-1,1,1}^{\ell,m} \shY[\ell,m-1]
	-\sum_{\ell=1}^{L}\sum_{m=-\ell}^{\ell}\imu b_{\ell,m} C_{\ell,m+1,1,-1}^{\ell,m}\shY[\ell,m+1]\right)\notag\\[1mm]
	&\qquad=-\frac{1}{\sqrt{2}}
	\left(\sum_{\ell=1}^{L}\sum_{m=-\ell-1}^{\ell-1}\imu b_{\ell,m+1}C_{\ell,m,1,1}^{\ell,m+1} \shY
	-\sum_{\ell=1}^{L}\sum_{m=-\ell+1}^{\ell+1}\imu b_{\ell,m-1} C_{\ell,m,1,-1}^{\ell,m-1}\shY\right)\notag\\[1mm]
	&\qquad= -\frac{1}{\sqrt{2}}
	\sum_{\ell=1}^{L}\sum_{m=-\ell}^{\ell}\imu\left(b_{\ell,m+1}C_{\ell,m,1,1}^{\ell,m+1}
	-b_{\ell,m-1} C_{\ell,m,1,-1}^{\ell,m-1}\right)\shY, 
\end{align*}
where we use \eqref{eq:cl,dl,betal} and \eqref{eq:ninecoeff}.
Let
\begin{align*}
	\eta_{\ell,m}^{(1)} := \left\{\begin{array}{ll}\imu\left(b_{\ell,m+1}C_{\ell,m,1,1}^{\ell,m+1}
	-b_{\ell,m-1} C_{\ell,m,1,-1}^{\ell,m-1}\right), & \ell=1,\dots,L, \; m=-\ell,\dots,\ell,\\
	0, & \ell=0,\; m = 0,
	\end{array}\right.
\end{align*}
we then obtain the equation
\begin{equation*}
	-\frac{1}{\sqrt{2}}\sum_{\ell=1}^{L}\sum_{m=-\ell}^{\ell}
	\left(b_{\ell,m}D_{+1,\ell,m}-b_{\ell,m}D_{-1,\ell,m}\right)
	= -\frac{1}{\sqrt{2}}\:S_{L}\bigl(\eta_{\ell,m}^{(1)}\bigr).
\end{equation*}
In a similar way,
\begin{equation*}
	-\frac{1}{\sqrt{2}}\imu \sum_{\ell=1}^{L}\sum_{m=-\ell}^{\ell}\left(b_{\ell,m}D_{+1,\ell,m}+b_{\ell,m}D_{-1,\ell,m}\right)
	= \frac{1}{\sqrt{2}}
	\sum_{\ell=1}^{L}\sum_{m=-\ell}^{\ell}\left(b_{\ell,m+1}C_{\ell,m,1,1}^{\ell,m+1}
	+b_{\ell,m-1} C_{\ell,m,1,-1}^{\ell,m-1}\right)\shY. \label{eq:difree1}
\end{equation*}
If letting
\begin{align*}
	\eta_{\ell,m}^{(2)} := \left\{\begin{array}{ll}	
\left(b_{\ell,m+1}C_{\ell,m,1,1}^{\ell,m+1}
	+b_{\ell,m-1} C_{\ell,m,1,-1}^{\ell,m-1}\right), & \ell=1,\dots,L,\; m=-\ell,\dots,\ell,\\
	0, & \ell=0,\; m=0,
	\end{array}\right.
\end{align*}
we then obtain
\begin{equation*}
	-\frac{1}{\sqrt{2}}\imu \sum_{\ell=1}^{L}\sum_{m=-\ell}^{\ell}\left(b_{\ell,m}D_{+1,\ell,m}+b_{\ell,m}D_{-1,\ell,m}\right)
	= \frac{1}{\sqrt{2}}\:S_{L}\bigl(\eta_{\ell,m}^{(2)}\bigr).
\end{equation*}

As
\begin{equation*}
	\sum_{\ell=1}^{L}\sum_{m=-\ell}^{\ell}b_{\ell,m}D_{0,\ell,m}
	= \sum_{\ell=1}^{L}\sum_{m=-\ell}^{\ell}\imu b_{\ell,m}C_{\ell,m,1,0}^{\ell,m}\shY,
\end{equation*}
we let
\begin{equation*}
	\eta_{\ell,m}^{(3)}:= \left\{\begin{array}{ll}
 \imu b_{\ell,m}C_{\ell,m,1,0}^{\ell,m}, & \ell=1,\dots,L,\; m=-\ell,\dots,\ell,\\
	0, & \ell=0,\; m=0,
	\end{array}\right.
\end{equation*}
one then has the representation
	$\sum_{\ell=1}^{L}\sum_{m=-\ell}^{\ell}b_{\ell,m}D_{0,\ell,m}
	= S_{L}(\eta_{\ell,m}^{(3)})$.
Thus,
\begin{align*}
	\bbS_{L}({a_{\ell,m}},{b_{\ell,m}})
	&= \begin{pmatrix}
		-\frac{1}{\sqrt{2}} \left(S_{L-1}\bigl(\nu^{(1)}_{\ell,m}\bigr) + S_{L+1}\bigl(\nu^{(2)}_{\ell,m}\bigr)\right)\\[1mm]
		-\frac{1}{\sqrt{2}} \left(S_{L-1}\bigl(\nu^{(3)}_{\ell,m}\bigr) + S_{L+1}\bigl(\nu^{(4)}_{\ell,m}\bigr)\right)\\[1mm]
		S_{L-1}(\nu_{\ell,m}^{(5)}) + S_{L+1}(\nu_{\ell,m}^{(6)})
	\end{pmatrix}
	+\begin{pmatrix}
		-\frac{1}{\sqrt{2}}S_{L}\bigl(\eta_{\ell,m}^{(1)}\bigr)\\[1mm]
		\frac{1}{\sqrt{2}}S_{L}\bigl(\eta_{\ell,m}^{(2)}\bigr)\\[1mm]
		S_{L}(\eta_{\ell,m}^{(3)})
	\end{pmatrix}\notag\\
	&= \begin{pmatrix}
		-\frac{1}{\sqrt{2}} \left(S_{L-1}\bigl(\nu^{(1)}_{\ell,m}\bigr) + S_{L+1}\bigl(\nu^{(2)}_{\ell,m}\bigr)+S_{L}\bigl(\eta_{\ell,m}^{(1)}\bigr)\right)\\[1mm]
		-\frac{1}{\sqrt{2}} \left(S_{L-1}\bigl(\nu^{(3)}_{\ell,m}\bigr) + S_{L+1}\bigl(\nu^{(4)}_{\ell,m}\bigr) - S_{L}\bigl(\eta_{\ell,m}^{(2)}\bigr)\right)\\[1mm]
		S_{L-1}(\nu_{\ell,m}^{(5)}) + S_{L+1}(\nu_{\ell,m}^{(6)}) + S_{L}(\eta_{\ell,m}^{(3)})
		\end{pmatrix},
\end{align*}
which then completes the proof.
\end{proof}

\subsection{Algorithms and Errors}\label{sec:complexity_error}
\subsubsection{Fast algorithms} In Algorithms~\ref{alg:favest.fwd} and \ref{alg:favest.adj}, we write down the algorithms for FwdVSHT and AdjVSHT from Theorems~\ref{thm:almblmviaFlm} and \ref{thm:SLvecviascal}. They achieve fast computation by FFTs for scalar spherical harmonics. In this paper, we use the NFFT package \cite{KeKuPo2009} for fast scalar spherical harmonic transforms on $\sph{2}$ which have the computational cost $\bigo{N\log \sqrt{N}}$ and $\bigo{M\log \sqrt{M}}$ for $N$ evaluation points and $M$ (Fourier) coefficients.
\begin{algorithm}[ht]
\SetAlgoNoLine
\KwIn{A sequence $\{T_1,\dots,T_N\} \subset\Rd[3]$, $N\geq2$ and a quadrature rule $\QN:=\{(w_i,\bx_i)\}_{i=1}^{N}$ on $\sph{2}$, and maximal degree $L$, $L\geq1$.}
\KwOut{Complex sequences of AdjVSHT $\dflm(T_k)$ and $\cflm(T_k)$ in \eqref{eq:fwdvsht2}, $\ell=1,\dots,L,\; m=-\ell,\dots,\ell$.}
\begin{itemize}
	\item[Step 1] Compute the AdjSHT $\flm\left(-T^{(1)}_{\cdot}+\imu T^{(2)}_{\cdot}\right)$, $\flm\left(T^{(1)}_{\cdot}+\imu T^{(2)}_{\cdot}\right)$ and $\flm\left(T^{(3)}_{\cdot}\right)$ for $\ell=0,\dots,L+1$, $m=-\ell,\dots,\ell$, by forward FFT for scalar spherical harmonics.
	\item[Step 2] Compute $\xi_{\ell,m}^{(i)}$, $i=1,2,\dots,6$ and $\mu_{\ell,m}^{(j)}$, $j=1,2,3$, for $\ell = 0,\dots,L+1$, $m=-\ell,\dots,\ell$.
	\item[Step 3] Compute $\dflm(T_{\cdot})$ and $\cflm(T_{\cdot})$ for $\ell=1,\dots,L,\; m=-\ell,\dots,\ell$, by \eqref{eq:almblmviaFlm}.
\end{itemize}
\caption{Forward {\fav}}
\label{alg:favest.fwd}
\end{algorithm}
\vskip -5mm
\begin{algorithm}[h]
\SetAlgoNoLine
\KwIn{Two complex sequences of coefficients $\{(a_{\ell,m}, b_{\ell,m}): \ell=1,\dots,L,\; m=-\ell,\dots,\ell\}$, and $M$ evaluation points $\{\bx_i\}_{i=1}^M$.}
\KwOut{Complex sequence of AdjVSHT $\bbS_{L}(a_{\ell,m},b_{\ell,m};\bx_i)$ in \eqref{eq:SLvec}, $i=1,\dots,M$.}
For each point $\bx_i$, $i=1,\dots,M$, do the following steps.
\begin{itemize}
	\item[Step 1] Compute $\nu_{\ell,m}^{(i)},\eta_{\ell,m}^{(j)}$, $i=1,\dots,6$, $j=1,2,3$, by \eqref{eq:nulm1}, \eqref{eq:nulm2} and \eqref{eq:etalm}.
	\item[Step 2] Evaluate $S_{L}(\nu_{\ell,m}^{(1)})$, $S_{L+1}(\nu_{\ell,m}^{(2)})$, $S_{L-1}(\nu_{\ell,m}^{(3)})$, $S_{L+1}(\nu_{\ell,m}^{(4)})$, $S_{L-1}(\nu_{\ell,m}^{(5)})$, $S_{L+1}(\nu_{\ell,m}^{(6)})$, $S_{L}(\eta_{\ell,m}^{(1)})$, $S_{L}(\eta_{\ell,m}^{(2)})$ and $S_{L}(\eta_{\ell,m}^{(3)})$, by adjoint FFT for scalar spherical harmonics.
	\item[Step 3] Compute $\bbS_{L}(a_{\ell,m},b_{\ell,m};\bx_i)$, by \eqref{eq:SLvecviascal}.
\end{itemize}
\caption{Adjoint {\fav}}
\label{alg:favest.adj}
\end{algorithm}
In Algorithm~\ref{alg:favest.fwd}, using the forward FFT for scalar spherical harmonics, Step~1 uses $3$ times scalar FFTs with degree up to $L+1$ and then has cost $\bigo{N\log \sqrt{N}}$ (assuming using $N=\bigo{L^2}$ points); Step~2 is the direct computation of $\xi_{\ell,m}^{(i)}$ and $\mu_{\ell,m}^{(j)}$ by \eqref{eq:xi_klm}, \eqref{eq:mu_klm} and \eqref{eq:ninecoeff} for degree $\ell$ at most $L+1$ and then has cost $\bigo{\sqrt{N}}$; Step~3 evaluates \eqref{eq:almblmviaFlm} using the results of Steps~1 and 2 and has cost $\bigo{1}$. Thus, evaluating the $\dflm(\{T_k\}_{k=1}^{N},\QN)$ and $\cflm(\{T_k\}_{k=1}^{N},\QN)$ incurs computational cost proportional to $N\log \sqrt{N}$. 
 In Algorithm~\ref{alg:favest.adj}, with $M=\bigo{L^2}$ evaluation points, Step~1 is computed by \eqref{eq:nulm1}, \eqref{eq:nulm2}, \eqref{eq:etalm} and \eqref{eq:cl,dl,betal}, \eqref{eq:ninecoeff} for degree up to $L+1$ has computational cost $\bigo{\sqrt{M}}$; Step~2 uses $8$ times adjoint FFTs for scalar spherical harmonics and then has computational cost $\bigo{M\log\sqrt{M}}$; Step~3 computes \eqref{eq:SLvecviascal} by the results of the previous steps has cost $\bigo{1}$. Thus, the {\fav} for the adjoint case has computational cost proportional to $M\log\sqrt{M}$. These analyses show that the computational complexity of the proposed algorithms is nearly linear. We then call the algorithms \emph{\textbf{Fa}st \textbf{Ve}ctor \textbf{S}pherical Harmonic \textbf{T}ransform} or \emph{{\fav}}, and call Algorithms~\ref{alg:favest.fwd} and \ref{alg:favest.adj} \emph{forward {\fav}} and \emph{adjoint {\fav}} respectively.

\subsubsection{Errors.} Let $T$ be a tangent field in $\vLp{2}{2}$.
The approximation error of FwdVSHT $\dflm(\{T(\bx_k)\}_{k=1}^{N},\QN)$ and $\cflm(\{T(\bx_k)\}_{k=1}^{N},\QN)$ for the divergence-free and curl-free coefficients $\dfco{T}$ and $\cfco{T}$ in \eqref{eq:dfco_cfco} depends on the approximation quality of quadrature rule $\{(w_k,\bx_k)\}_{k=1}^{N}$ for integrals on the sphere and the smoothness of the tangent field $T$. Given a tangent field, the choice of quadrature rule is the key to reducing the approximation error. A quadrature rule $\{(w_i,\bx_i)\}_{i=1}^N$ is called \emph{exact for polynomials} of degree $L$, $L\geq0$ if for all spherical polynomials $p$ of degree at most $L$,
\begin{equation*}
	\sum_{i=1}^N w_i p(\bx_i) = \int_{\sph{2}}p(\bx)\ds,
\end{equation*}
see e.g. \cite{HeSlWo2015}. Algorithm~\ref{alg:favest.fwd} using a quadrature rule that is exact for scalar spherical polynomial of degree $L$ has the approximation error $CL^{-s}$ for $\dsh^{*}T$ lying in Sobolev space $H^s(\sph{2})$ (of scalar spherical functions) and $s>1$, where the constant $C$ depends only on the Sobolev norm of the function $\dsh^{*}T$. The order $L^{-s}$ is optimal, as a consequence of \cite{HeSl2005Optimal,HeSl2005Worst,HeSl2006Cubature}.
On the other hand, Algorithm~\ref{alg:favest.adj} can evaluate FwdVSHT with zero-loss. Its approximation for the expansion \eqref{eq:expan} is equal to the truncation error for the vector spherical harmonic expansion, as determined by
\begin{equation*}
	\biggl\|\sum_{\ell=L+1}^{\infty}\sum_{m=-\ell}^{\ell}\left(a_{\ell,m}\dsh + b_{\ell,m}\csh\right)\biggr\|_{L_2}^2
	=
	\sum_{\ell=L+1}^{\infty}\sum_{m=-\ell}^{\ell}\left(|a_{\ell,m}|^2 + |b_{\ell,m}|^2\right),
\end{equation*}
where we have used the orthogonality of vector spherical harmonics.

\subsection{Software Description}\label{sec:software}
We provide the software package in Matlab for {\fav} which includes Matlab demo and routines for {\fav} and the routines for numerical examples in the next section.
The {\fav} package can be downloaded from GitHub at \url{https://github.com/mingli-ai/FaVeST}. It has been tested in Matlab environment in operating systems including Ubuntu 16.04.6, macOS High Sierra, macOS Mojave, Windows 7, Windows 8 and Windows 10.
In the Matlab library repository, the main m-files include \textbf{FaVeST$\_$fwd.m} and \textbf{FaVeST$\_$adj.m}, corresponding to Algorithms~\ref{alg:favest.fwd} and \ref{alg:favest.adj} respectively. Inside these two functions, the package NFFT\footnote{\url{https://www-user.tu-chemnitz.de/~potts/nfft}} \cite{KeKuPo2009} is used to run the scalar FFTs on the sphere. (The NFFT is the only package requested by the {\fav} package.) There are four inputs for the function \textbf{FaVeST$\_$fwd.m}: $T,L,X,w$, where $T$ is a tangent field sampled at the point set $X$, $L$ is the highest degree of spherical harmonics, $X$ is the set of quadrature nodes and $w$ is the set of quadrature weights. The output includes the two sequences of divergence-free and curl-free coefficients $\dflm(T_{\cdot})$ and $\cflm(T_{\cdot})$ for degree $\ell=1,\dots,L$ and $m=-\ell,\dots,\ell$ as given in  \eqref{eq:fwdvsht2}. The input of \textbf{FaVeST$\_$adj.m} contains two finite sequences $a_{\ell,m}$ and $b_{\ell,m}$ (with $\ell\leq L$ for some $L\geq1$ and $m=-\ell,\dots,\ell$) as the coefficients for divergence-free and curl-free parts, and a point set $X$ for evaluation. The output is the AdjVSHT in \eqref{eq:SLvec} for $a_{\ell,m}$ and $b_{\ell,m}$.


\section{Numerical Examples}
\label{sec:numer}
In this section, we show numerical examples for verification of the performance of the proposed {\fav} algorithm. We start from the description of two types of polynomial-exact quadrature rules used in the experiments, and then present the examples of tangent fields on the sphere. We show the reconstruction and error in orthographic projection on the sphere for {\fav} using both kinds of point sets. We also show the CPU computational time of {\fav} which evaluates a tangent field for degree up to $2,250$ and at $10$ million spherical points.

\subsection{Quadrature Rules}
We use two types of point sets on $\sph{2}$ in the experiments, as follows.
\begin{enumerate}
  \item \emph{Gauss-Legendre tensor product rule} (GL) \cite{HeWo2012}. The Gauss-Legendre tensor (product) rule is a (polynomial-exact but not equal area) quadrature rule $\QN:=\{(w_i,\bx_i)\}_{i=1}^{N}, i=0,\ldots,N\}$ on the sphere generated by the tensor product of the Gauss-Legendre nodes on the interval $[-1,1]$ and equi-spaced nodes on the longitude with non-equal weights. To be exact for polynomials of degree $L$, one needs around $2L^2$ GL points. Figure~\ref{fig:QN}(a) shows $N=512$ GL points for degree $L=16$. An alternative GL was developed by \cite{graf2011computation,graf2013efficient,KeKuPo2020} with smaller number of points.
  \item \emph{Symmetric spherical $t$-designs} (SD) \cite{Womersley_ssd_URL}. The symmetric spherical $t$-design is a (polynomial-exact) quadrature rule $\QN:=\{(w_i,\bx_i)\}_{i=1}^{N}, i=0,\ldots,N\}$ on the sphere $\sph{2}$ with equal weights $w_i=1/N$. The points are almost uniformly (or in formal definition, quasi-uniform\footnote{The quasi-uniformity describes the distribution of a sequence of point sets. A sequence of point sets on the sphere is called quasi-uniform if the quotient of the covering and separation radii of each point set is bounded, see e.g. \cite{SoWaWuYu2019,MhNaWa2001,Womersley_ssd_URL}.}) distributed on the sphere. To be exact for polynomials of degree $L$, one needs to use the symmetric spherical $t$-design for $t=L$ with around $L^2/2$ points. Figure~\ref{fig:QN}(b) shows $N=498$ SD points for degree $L=31$.
\end{enumerate}
\begin{figure}[th]
\begin{minipage}{\textwidth}
  \centering
  \begin{minipage}{0.3\textwidth}
  \centering
  \includegraphics[trim = 0mm 0mm 0mm 0mm, width=0.81\textwidth]{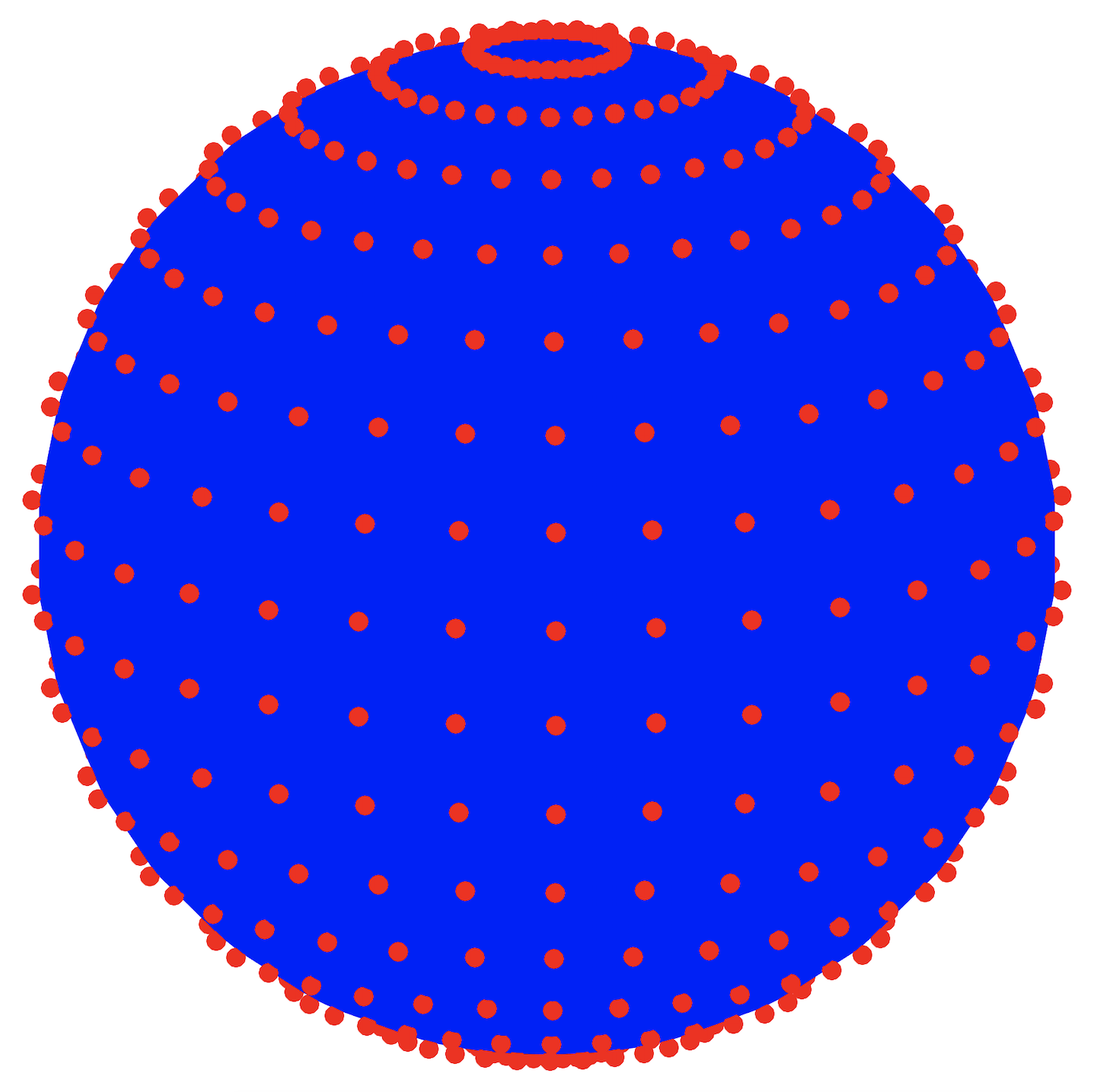}
  \subcaption{GL, $N=512$}\label{fig:GL}
  \end{minipage}\hspace{1mm}
  \begin{minipage}{0.3\textwidth}
  \centering
  \includegraphics[trim = 0mm 0mm 0mm 0mm, width=0.81\textwidth]{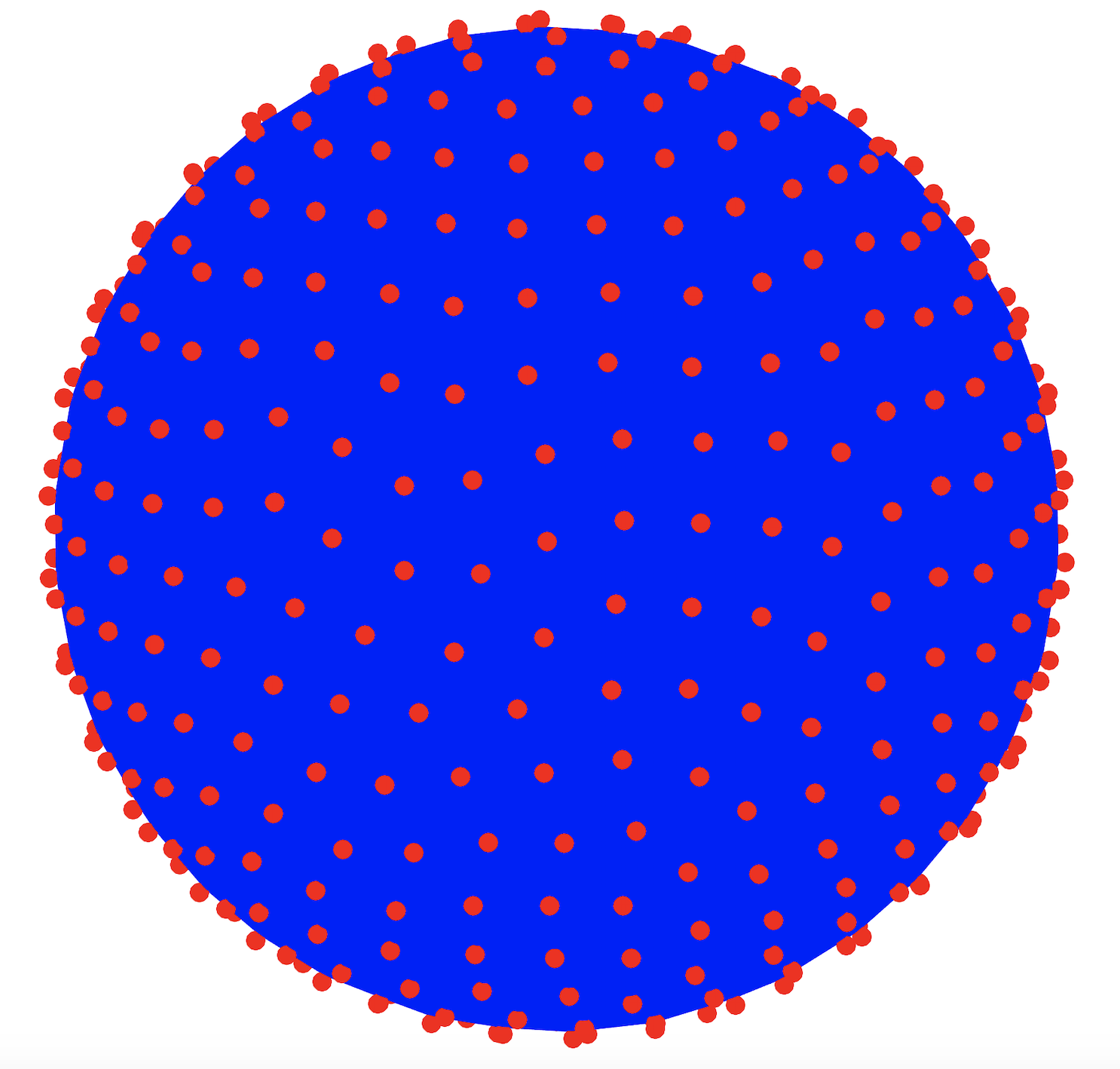}
  \subcaption{SD, $N=498$}\label{fig:SD}
  \end{minipage}
\end{minipage}
\vspace{-5mm}
\begin{minipage}{0.9\textwidth}
\caption{Point sets $\sph{2}$. (a) Nodes of Gauss-Legendre tensor rule (GL). (b) Nodes of symmetric spherical designs (SD).}
\label{fig:QN}
\end{minipage}
\end{figure}

\subsection{Tangent Fields}
To verify our theoretical results in Section 3, we use three types of simulated tangent fields as provided in \cite{FuWr2009}. All these tangent fields are generated using \emph{stream function} and \emph{velocity potential} so that we can easily split the divergence-free and curl-free parts of the field. Let $s$ and $v$ be the stream function and velocity potential, then, each of the tangent fields can be represented by
\begin{equation*}
  T=\underbrace{\mathbf{L}s}_{f^{\rm div}}+\underbrace{\nabla_{*}v}_{f^{\rm curl}}.
\end{equation*}
Recall from Section~\ref{sec:vsh} that $\mathbf{L}$ and $\nabla_{*}$ denote the surface curl and surface gradient, and $\mathbf{L}s$ and $\nabla_{*}v$ are divergence-free and curl-free.
We define the three tangent fields, as follows.

\paragraph{Tangent Field A} The stream function and velocity potential for this field are linear combinations of spherical harmonics. They can be used to generate realistic synoptic-scale meteorological wind fields.
The stream function is defined by
\begin{equation}\label{RosbyHaurwitz}
s_1(\PT{x})=-\frac{1}{\sqrt{3}}Y_{1,0}(\PT{x})+\frac{8\sqrt{2}}{3\sqrt{385}}Y_{5,4}(\PT{x}),
\end{equation}
which is known as a Rosby--Haurwitz wave and is an analytic solution to the nonlinear
barotropic vorticity equation on the sphere \cite[pp.~453--454]{Holton1973}. In  \cite{Williamson1992}, $s_1$ was used as the
initial condition for the shallow water wave equations on the sphere.
The velocity potential is given by
\begin{equation*}
  v_1(\PT{x})=\frac{1}{25}(Y_{4,0}(\PT{x})+Y_{6,-3}(\PT{x})).
\end{equation*}
Note that we can choose spherical harmonics with different degrees and different coefficients in the above formula. Here, we have used the same setting as \cite{FuWr2009}.

\paragraph{Tangent Field B} This field still uses the Rosby--Haurwitz wave (\ref{RosbyHaurwitz}) as the
stream function. But the velocity potential is a linear combination of compactly supported functions:
\begin{equation*}
  v_2(\PT{x})=\frac{1}{8}f(\PT{x};5,\pi/6,0)-\frac{1}{7}f(\PT{x};3,\pi/5,-\pi/7)+\frac{1}{9}f(\PT{x};5,-\pi/6,\pi/2)-\frac{1}{8}f(\PT{x};3,-\pi/5,\pi/3),
\end{equation*}
where
\begin{equation*}
  f(\PT{x};\sigma,\theta_c,\lambda_c)=\frac{\sigma^3}{12}\sum_{j=0}^{4}(-1)^j\left(
\begin{array}{c}
4\\
j
\end{array} \right)\left|r-\frac{(j-2)}{\sigma}\right|^3.
\end{equation*}

\emph{Tangent Field C}. Let $\PT{x}_c\in S^2$ in spherical coordinates $(\theta_c,\lambda_c)$, and $t=\PT{x}\cdot \PT{x}_c$ and $a=1-t$. Define
\begin{equation*}
  g(\PT{x};\theta_c,\lambda_c)=-\frac{1}{2}((3t+3\sqrt{2}a^{3/2}-4)+(3t^2-4t+1)\log(a)+(3t-1)a\log(\sqrt{2a}+a)).
\end{equation*}
The stream function for this tangent field is given by
\begin{equation*}
  s_3(x)=\int_{-\pi/2}^{\theta}\sin^{14}(2\xi)\mathrm{d}\xi-3g(x;\pi/4,-\pi/12),
\end{equation*}
where $\theta$ denotes the latitudinal coordinate of $\PT{x}$.
With $g$, the velocity potential is given by
\begin{equation*}
  v_3(\PT{x})=\frac{5}{2}g(\PT{x};\pi/4,0)-\frac{7}{4}g(\PT{x};\pi/6,\pi/9)-\frac{3}{2}g(\PT{x};5\pi/16,\pi/10).
\end{equation*}

\subsection{Reconstruction and Errors}
The left columns of Figures~\ref{fig:reconstruction_gl} and \ref{fig:reconstruction_sd} present the tangent fields sampled at $N=1922$ GL points and $N=1894$ SD points, respectively.
The middle columns of Figure \ref{fig:reconstruction_gl} and \ref{fig:reconstruction_sd} show the reconstructed field $T^{\rm rec}$ for the tangent field $T$ with $N=1922$ GL points and $N=1849$ SD points for evaluation. The corresponding point-wise error $T-T^{\rm rec}$ (at the sampling points) is displayed in the right columns. Here, the length and arrow indicate the scalar value and direction of the field at a spherical point. We observe that the reconstruction performance of the {\fav} is excellent: the relative error for the reconstruction is small compared with the magnitude of the original field in both GL and SD cases. The reconstructed field restores the direction and scalar value of the original field in high precision, especially when the field is sufficiently smooth. The error fields for Tangent Fields B and C show some local features of the evaluation of {\fav}: the smoother part has a smaller error, which may be partially interpreted by the localization approximation behaviour of spherical polynomials, see, e.g. \cite{Wang2016, WaLeSlWo2017, LeMh2008, LeSlWaWo2017, WaSlWo2018}.

Table~\ref{tab:relative-err} reports the relative $L_2$-errors $\|T-T^{\rm rec}\|_2/\|T\|_2$ for the reconstruction with degree $L$ up to $150$ in the cases of GL and SD. We observe that the {\fav}s with GL and SD exhibit similar error orders for each tangent field.
These results illustrate that the {\fav} with a polynomial-exact quadrature rule is precise. In both GL and SD cases, for each degree, the error increases as the smoothness of the field decreases (from A to C). It means that the accuracy of the algorithm improves as the smoothness of the tangent field reduces. The error here is the superposition of the errors for forward {\fav} and adjoint {\fav}. They are determined by the approximation error of numerical integration by quadrature rule and the truncation error of vector spherical harmonic expansion, both of which increase as the smoothness of the tangent field reduces, see, e.g. \cite{Brauchart_etal2015, BrSaSlWo2014, FrGeSc1998}.

We use the same example to show the error incurred by repeated use of {\fav}. In the experiment, we repeat the forward and adjoint transforms twice to obtain the sequence $T_0\to \{\widehat{T}_0,\widetilde{T}_0\}\to T_1\to \{\widehat{T}_1,\widetilde{T}_1\}\to T_2$. Here $T_0$ is an original field on the sphere with the vector spherical harmonic coefficients $\{\widehat{T}_0,\widetilde{T}_0\}$; $T_1$ is the reconstructed field from the coefficients of $T_0$ via adjoint {\fav}. $T_2$ is the adjoint {\fav} for the coefficients $\{\widehat{T}_1,\widetilde{T}_1\}$ of the field $T_1$. 
In Table~\ref{tab:err repeated transforms}, we evaluate the errors for the above transforms on the tangent fields A, B, C over GL and SD points for vector spherical harmonics up to degree $40$ respectively, by computing the infimum norms $\|T_1-T_0\|_{\infty}$, $\|T_2-T_0\|_{\infty}$, $\|T_2-T_1\|_{\infty}$, and $\max\bigl\{\|\widehat{T}_1-\widehat{T}_0\|_{\infty},\|\widetilde{T}_1-\widetilde{T}_0\|_{\infty}\bigr\}$. 
For all the three fields, $\|T_1-T_0\|_{\infty}$ and $\|T_2-T_0\|_{\infty}$ are at the same order, and the errors are bigger as the original tangent field becomes rougher. This is due to that they contain the projection error of the spherical harmonic approximation to the tangent field.
 $\|T_2-T_1\|_{\infty}$ is at the order $10^{-12}$ which is the same to that of the scalar spherical harmonic transforms NFFT that we have used in {\fav}. This error indeed comes from the error of the NFFT since $T_1$ and $T_2$ are the reconstructed fields that are expanded by vector spherical harmonics at the same degree. The error for the two sets of coefficients $\max\bigl\{\|\widehat{T}_1-\widehat{T}_0\|_{\infty},\|\widetilde{T}_1-\widetilde{T}_0\|_{\infty}\bigr\}$ is at the order equal or less than $10^{-12}$, which is the analogue to that of $\|T_2-T_1\|_{\infty}$ dependent upon the NFFT's precision.
\begin{table}[t]
\centering
\begin{minipage}{0.97\textwidth}
\centering
\begin{tabular}{l*{12}{c}c}
\toprule
&$\mbox{Quadrature}$  &   $L=10$    &   $L=30$   &   $L=50$    &  $L=100$     &   $L=120$   &   $L=150$    \\
\midrule
 \multirow{2}{*}{Tangent Field A}& GL  &  8.6133e-12 &  4.3287e-12 &  3.1993e-12  & 2.6626e-12  & 2.5678e-12  & 2.4932e-12\\
   & SD & 5.3367e-12  & 3.2721e-12  & 2.9385e-12 &  2.5873e-12 &  4.0959e-12  & 1.4713e-10 \\
 \midrule
 \multirow{2}{*}{Tangent Field B}& GL    &7.5102e-02  & 2.9206e-03  & 6.5037e-04  & 1.0389e-04  & 7.1028e-05 &  3.5907e-05 \\
                          & SD &  7.1155e-02 &  3.0262e-03  & 6.9277e-04  & 1.1280e-04  & 7.6203e-05  & 4.0761e-05 \\
 \midrule
 \multirow{2}{*}{Tangent Field C}& GL    &2.6919e-01  & 5.7499e-03  & 1.9874e-03  & 4.7176e-04  & 3.3860e-04 &  2.5335e-04 \\
                          & SD  & 1.7225e-01  & 5.6764e-03 &  2.1336e-03 &  5.5568e-04 &  4.1357e-04 &  2.4919e-04 \\
 \bottomrule
\end{tabular}
\end{minipage}
\begin{minipage}{0.8\textwidth}
\vspace{3mm}
\caption{Relative $L_{2}$-errors $\|T-T^{\rm rec}\|_2/\|T\|_2$ in GL and SD cases for various numbers of nodes.}
\label{tab:relative-err}
\end{minipage}
\end{table}

\begin{table}[t]
\centering
\begin{minipage}{0.97\textwidth}
\centering
\begin{tabular}{l*{12}{c}c}
\toprule
&$\mbox{Quadrature}$  &   $\|T_1-T_0\|_{\infty}$    &   $\|T_2-T_0\|_{\infty}$   &   $\|T_2-T_1\|_{\infty}$    &  $\max\bigl\{\|\widehat{T}_1-\widehat{T}_0\|_{\infty},\|\widetilde{T}_1-\widetilde{T}_0\|_{\infty}\bigr\}$      \\
\midrule
 \multirow{2}{*}{Tangent Field A}& GL  &  2.0874e-12 &  4.0647e-12 &  2.0074e-12  & 2.9352e-12  \\
   & SD & 4.0994e-12  & 8.2248e-12  & 4.1511e-12 &  2.9310e-13 \\
 \midrule
 \multirow{2}{*}{Tangent Field B}& GL   & 1.4946e-03  & 1.4946e-03  & 2.8605e-12  & 2.9400e-12   \\
                        & SD &  2.0078e-03 &  2.0078e-03  & 4.5585e-12  & 3.0065e-13  \\
 \midrule
 \multirow{2}{*}{Tangent Field C}& GL    & 1.9968e-02  & 1.9968e-02& 1.2819e-11  & 3.0886e-12 \\
                          & SD  & 3.8045e-02  & 3.8045e-02 &  5.6859e-12 &  6.1313e-13  \\
 \bottomrule
\end{tabular}
\end{minipage}
\begin{minipage}{0.9\textwidth}
\vspace{3mm}
\caption{Errors when using forward and adjoint FaVeST repeatedly: $T_0\to \bigl\{\widehat{T}_0,\widetilde{T}_0\bigr\}\to T_1\to \bigl\{\widehat{T}_1,\widetilde{T}_1\bigr\}\to T_2$, for the tangent fields A, B, C over GL and SD points, respectively.}
\label{tab:err repeated transforms}
\end{minipage}
\end{table}

\subsection{Computational Complexity}
\begin{figure}[htbp!]
\centering
\begin{minipage}{\textwidth}
  \subcaptionbox{Reconstruction of Tangent Field A with GL}
  {\includegraphics[width=0.31\linewidth]{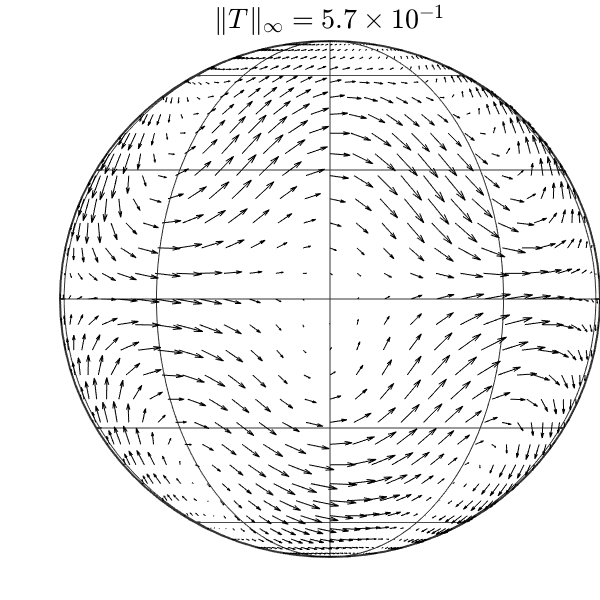}\quad
    \includegraphics[width=0.31\linewidth]{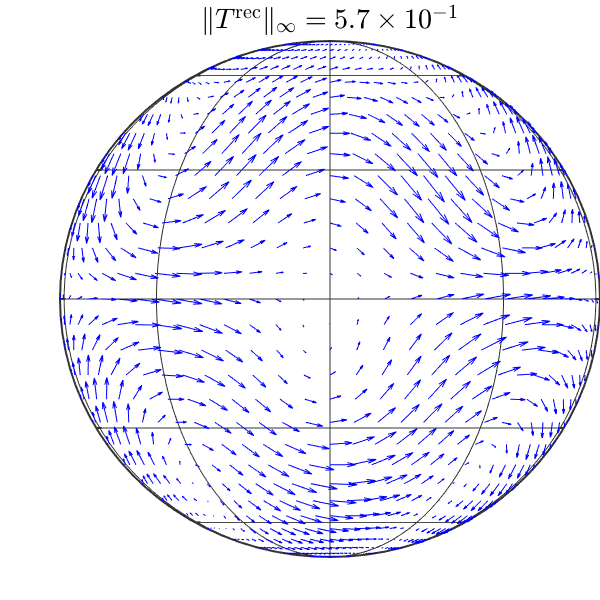}\quad
    \includegraphics[width=0.31\linewidth]{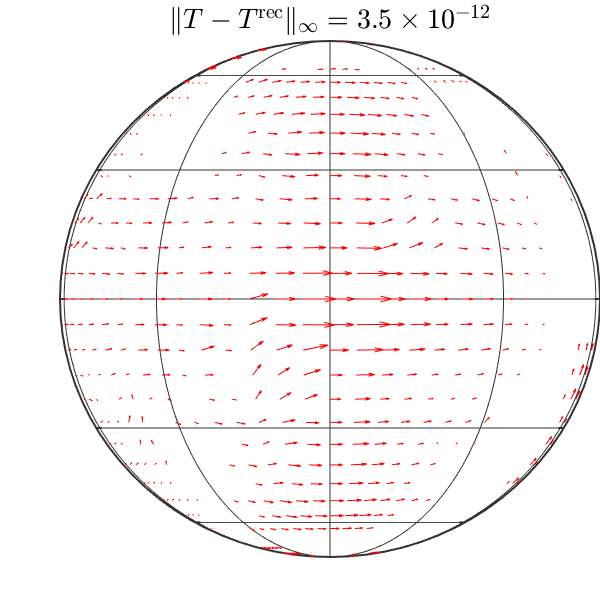}}\\[4mm]
    \subcaptionbox{Reconstruction of Tangent Field B with GL}
  {\includegraphics[width=0.31\linewidth]{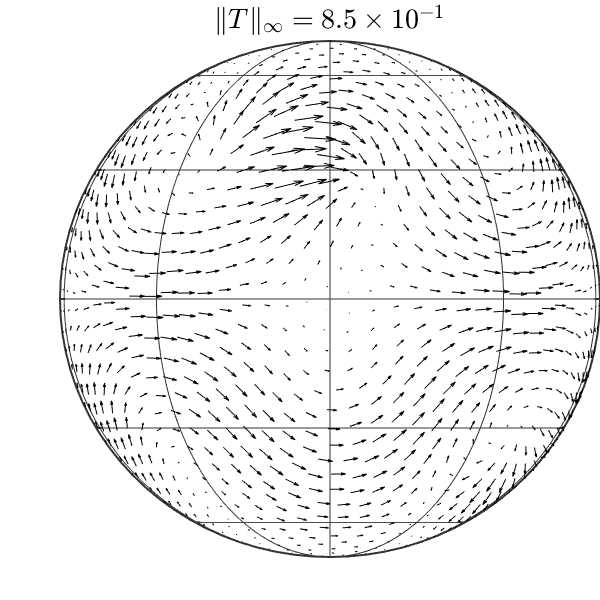}\quad
    \includegraphics[width=0.31\linewidth]{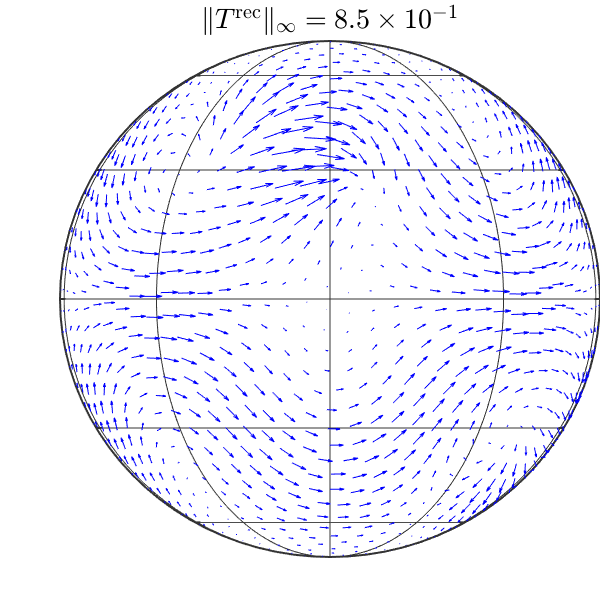}\quad
    \includegraphics[width=0.31\linewidth]{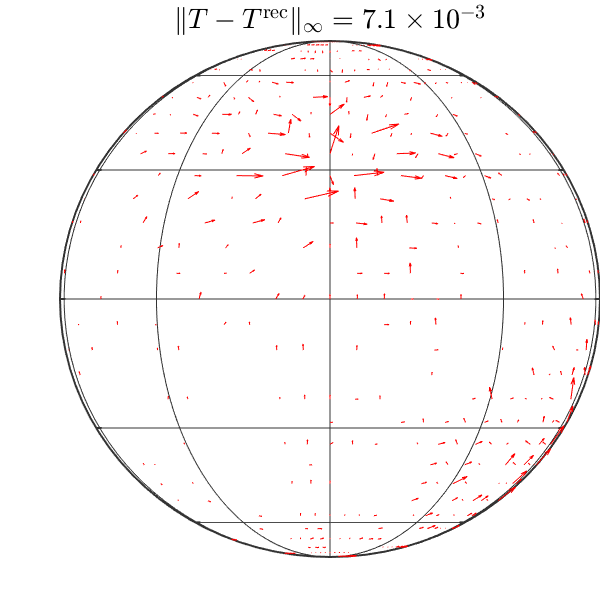}}\\[4mm]
    \subcaptionbox{Reconstruction of Tangent Field C with GL}
  {\includegraphics[width=0.31\linewidth]{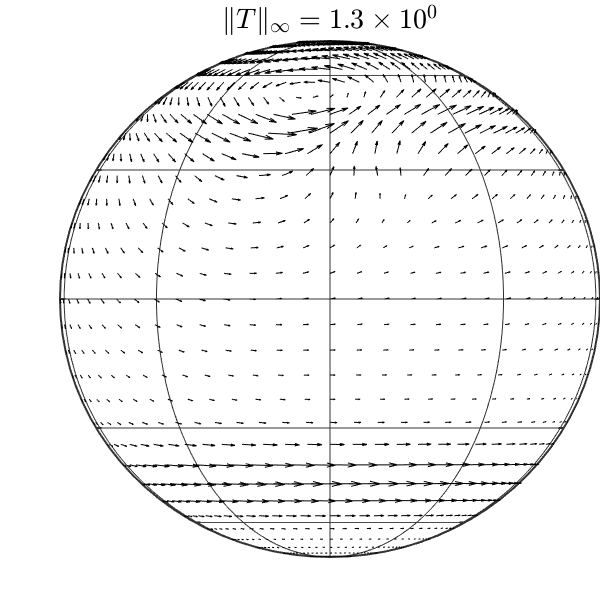}\quad
    \includegraphics[width=0.31\linewidth]{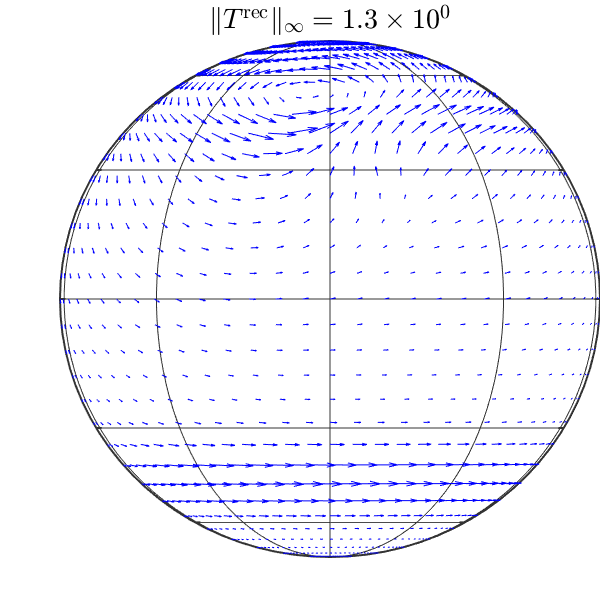}\quad
    \includegraphics[width=0.31\linewidth]{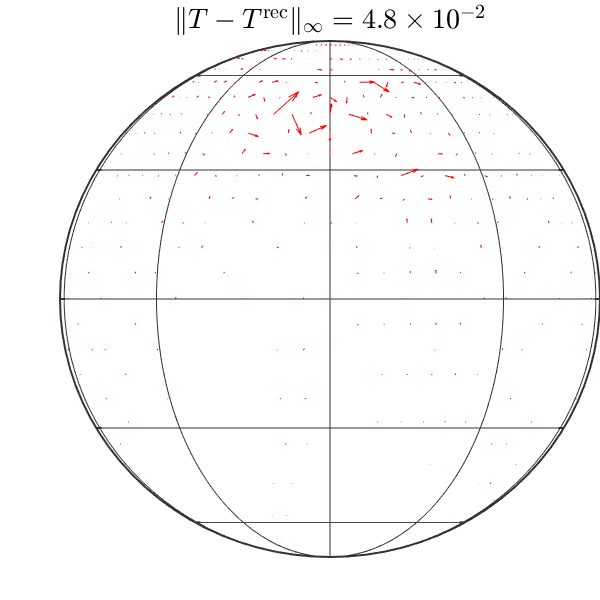}}
\end{minipage}
\vspace{3mm}
\begin{minipage}{0.9\textwidth}
    \caption{Visualization of reconstructed tangent field and error by {\fav} with quadrature rule of Gauss-Legendre tensor (GL). The first and the second columns show the target tangent field $T$ and reconstructed field $T^{\rm rec}$. The third column shows the error $T-T^{\rm rec}$. All plots are the orthographic projections of the fields evaluated at $N = 1922$ GL nodes. The normalized maximum norms for $T$, $T^{\rm rec}$ and $T-T^{\rm rec}$ are displayed in the titles.}\label{fig:reconstruction_gl}
\end{minipage}
\end{figure}

\begin{figure}[htbp!]
\centering
\begin{minipage}{\textwidth}
 \subcaptionbox{Reconstruction of Tangent Field A with SD}
{\includegraphics[width=0.31\linewidth]{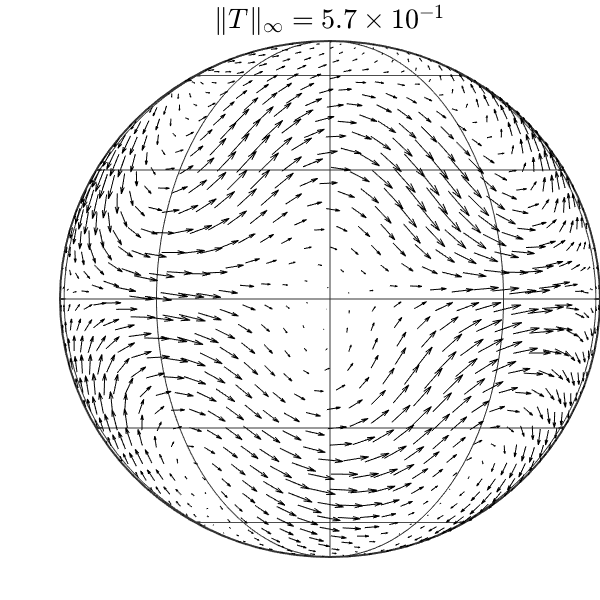}\quad
    \includegraphics[width=0.31\linewidth]{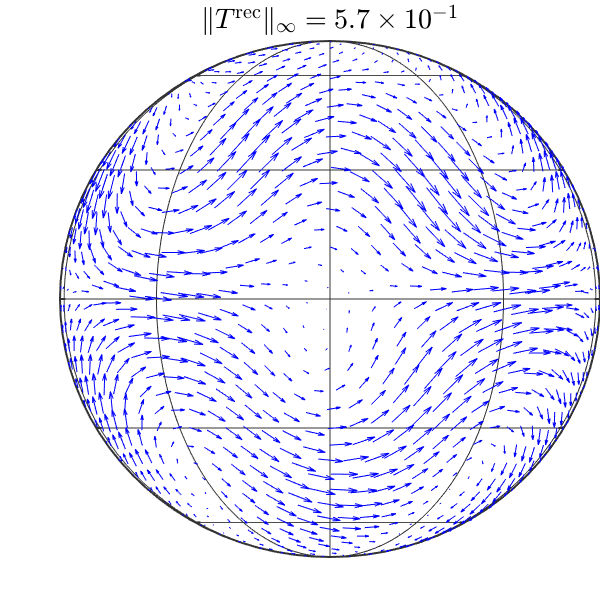}\quad
    \includegraphics[width=0.31\linewidth]{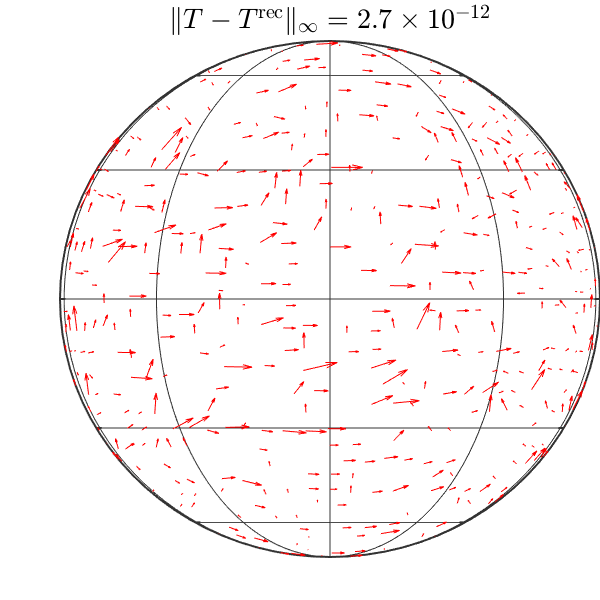}}\\[4mm]
    \subcaptionbox{Reconstruction of Tangent Field B with SD}
  {\includegraphics[width=0.31\linewidth]{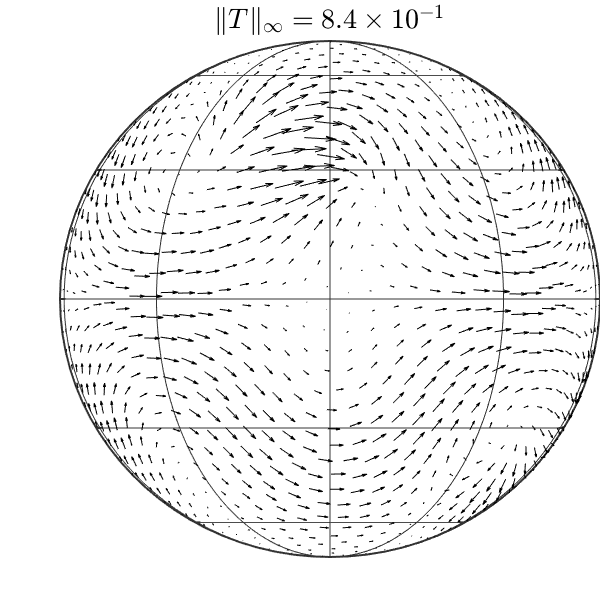}\quad
    \includegraphics[width=0.31\linewidth]{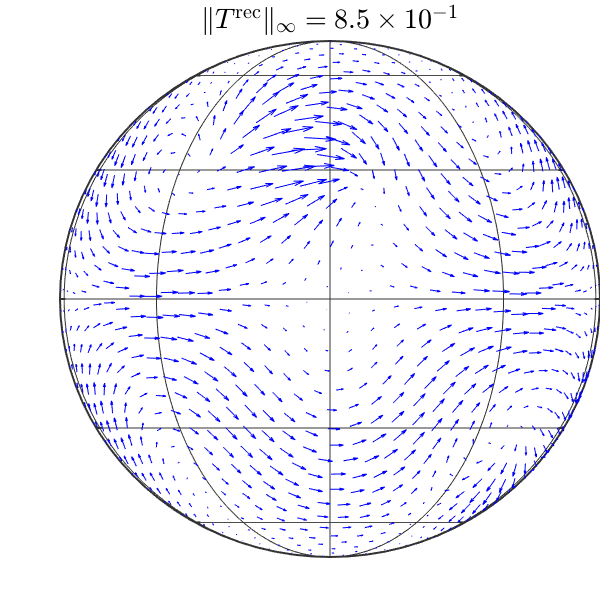}\quad
    \includegraphics[width=0.31\linewidth]{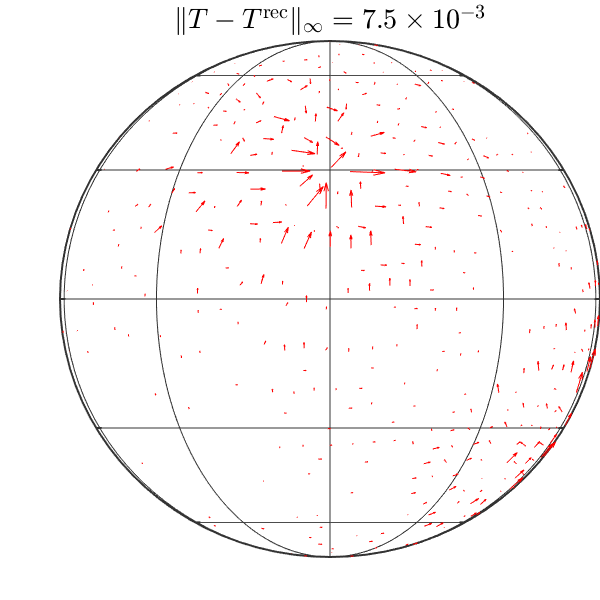}}\\[4mm]
    \subcaptionbox{Reconstruction of Tangent Field C with SD}
  {\includegraphics[width=0.31\linewidth]{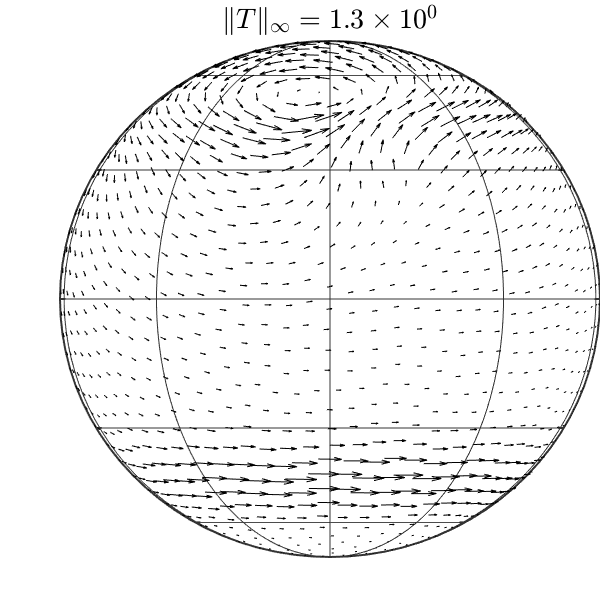}\quad
    \includegraphics[width=0.31\linewidth]{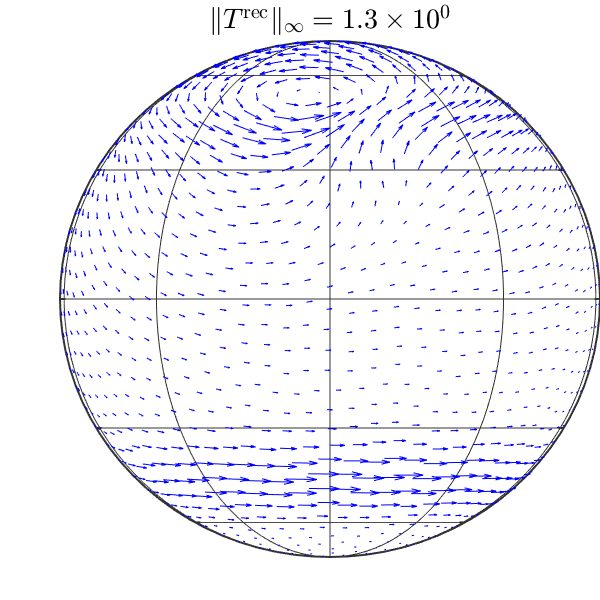}\quad
    \includegraphics[width=0.31\linewidth]{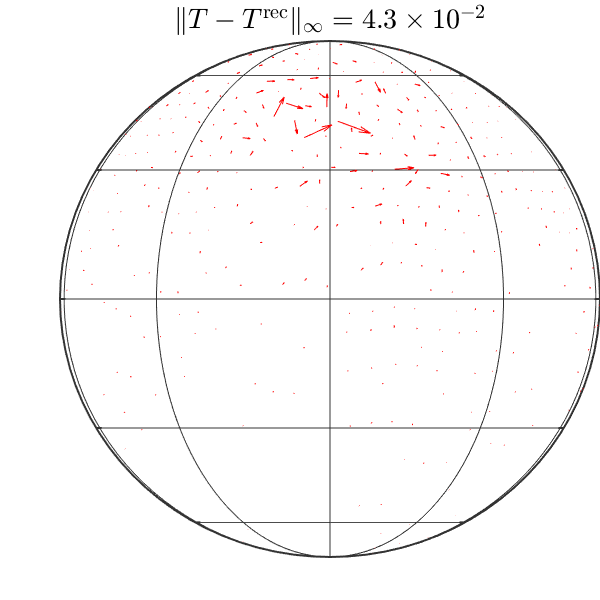}}
\end{minipage}
\begin{minipage}{0.9\textwidth}
    \caption{Visualization of reconstructed tangent field and error by {\fav} with quadrature rule of symmetric spherical $t$-designs (SD). The first and the second columns show the target tangent field $T$ and reconstructed field $T^{\rm rec}$. The third column shows the error $T-T^{\rm rec}$. All plots are the orthographic projections of the fields evaluated at $N = 1,849$ SD nodes. The normalized maximum norms for $T$, $T^{\rm rec}$ and $T-T^{\rm rec}$ are displayed in the titles.}\label{fig:reconstruction_sd}
\end{minipage}
\end{figure}

To test the time complexity of {\fav}, we carry out experiments with different numbers of GL points: let $L=250k+500$, $k=0,1,\ldots,8$, which corresponds to
$N_k\approx2L_k^2$ GL nodes for FwdVSHT, and $M_k=L_k^2+L_k$ coefficients for AdjVSHT. The CPU time consumption by {\fav}s is reported in Table~\ref{tab:timecost}. As indicated by the quotient ratios (in the round brackets) of the CPU times for degrees $L_k$ and $L_{k-1}$, the computational time is almost proportional to $N_k$ and $M_k$ for forward and adjoint {\fav}s. This is also confirmed by Figure~\ref{fig:timecost}.
 Figure~\ref{fig:timecost}(a) shows the trend of CPU time changing with the increase of the number of GL points for forward {\fav}. 
 Figure~\ref{fig:timecost}(b) shows the trend of CPU time varying with the increase of coefficient number for adjoint {\fav}. The fitting curves for the forward and adjoint {\fav} cases have the powers (up to a constant) $N^{1.2}$ and $M^{1.2}$, which means that both the forward and adjoint {\fav}s have near linear computational complexity. This is consistent with the analysis in Section~\ref{sec:complexity_error}.
 \begin{table}[h]
\centering
\begin{minipage}{\textwidth}
\centering
\small
\begin{tabular}{l*{11}{c}c}
\toprule
$L$      & 250  &   500       &  750       &   1000    &  1250      &   1500     &   1750    &  2000    & 2250  \\
\midrule
 $N$          &   126,002     &   502,002    &  1,128,002    &  2,004,002  &   3,130,002  &   4,506,002   &  6,132,002 &  8,008,002 & 10,134,002\\

 $t^{\mathrm{fwd}}$   & 0.40(4.56)  & 1.81(2.33) & 4.22(1.73) & 7.2903(2.08)  & 15.16(1.47)  &  22.30(1.47) & 32.87(1.26)  & 41.51(1.51)  & 62.74 \\\midrule

 $M$   & 63,000  & 251,000 & 564,000 &  1,002,000  & 1,565,000 & 2,253,000 &   3,066,000 & 4,004,000  & 5,067,000\\

 $t^{\mathrm{adj}}$   & 0.83(4.64)  & 3.87(2.46) & 9.52(1.88) & 17.89(1.90) & 33.97(1.48) & 50.32(1.47) & 73.94(1.32) & 97.32(1.49) & 144.70\\
\bottomrule
\end{tabular}
\end{minipage}
\begin{minipage}{0.9\textwidth}
\vspace{3mm}
\caption{Forward {\fav} CPU time $t^{\rm fwd}$ v.s. number of points and adjoint {\fav} CPU time $t^{\rm adj}$ v.s. number of coefficients. For $L=L_k=250k+500$, $k=0,1,\dots,8$, forward {\fav} uses Gauss-Legendre tensor rule which has $N=N_k\approx 2L_k^2$ nodes, and adjoint {\fav} uses $M=M_k=L_k^2+L_k$ coefficients. The numbers inside the brackets are the ratios $\frac{t^{\mathrm{fwd}}(N_{k})}{t^{\mathrm{fwd}}(N_{k-1})}$ and $\frac{t^{\mathrm{adj}}(M_{k})}{t^{\mathrm{adj}}(M_{k-1})}$. The numerical test is run under Intel Core i7-6700 CPU @ 3.40GHz with 16GB RAM in Windows 10.}
\label{tab:timecost}
\end{minipage}
\end{table}
\begin{figure}[htbp!]
\begin{minipage}{\textwidth}
	\centering
\begin{minipage}{\textwidth}
  \centering
  \begin{minipage}{0.46\textwidth}
  \centering
  \includegraphics[trim = 0mm 0mm 0mm 0mm, width=\textwidth]{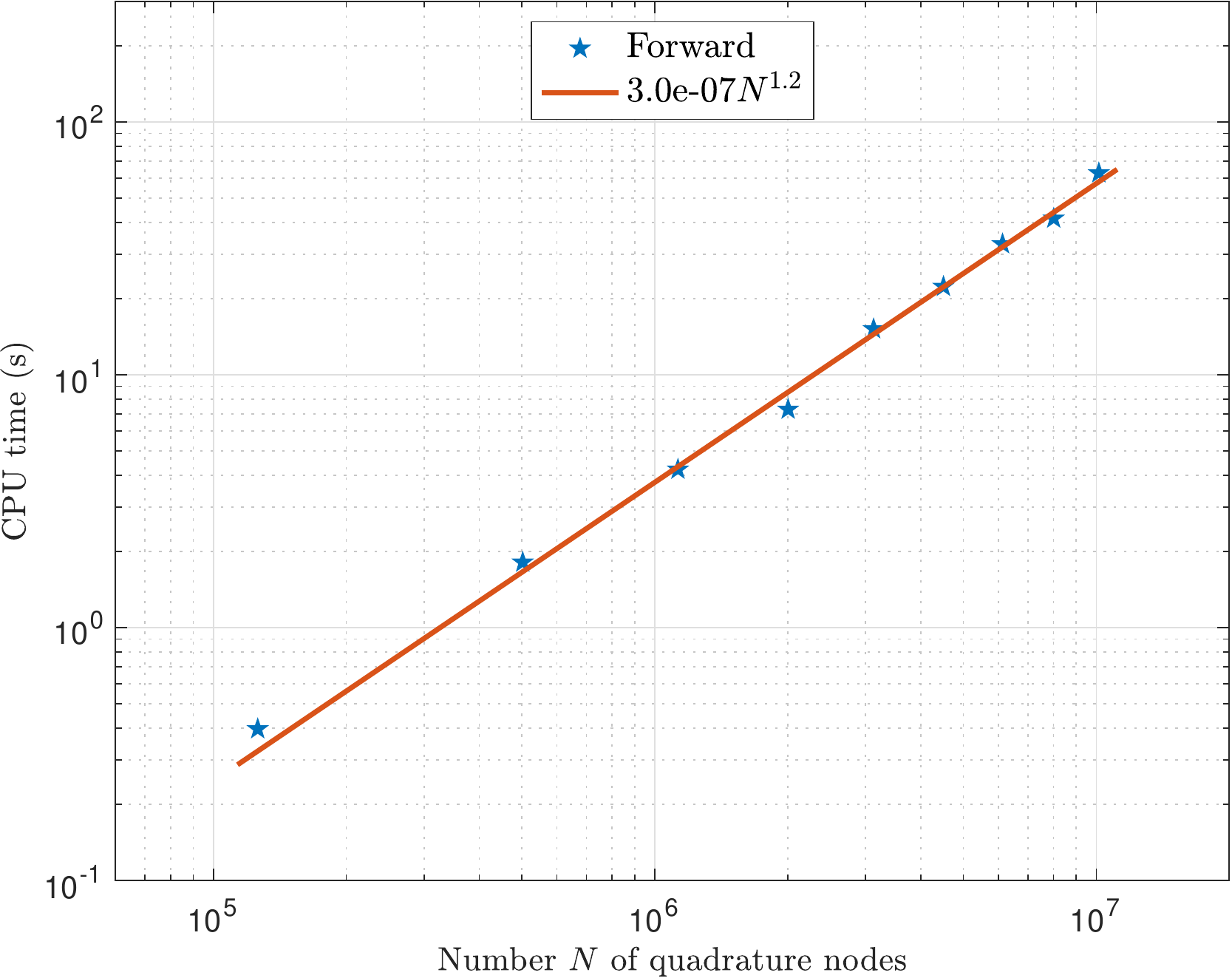}\\
  \subcaption{Forward {\fav}}\label{fig:time_favest_fwd}
  \end{minipage}\hspace{8mm}
  \begin{minipage}{0.46\textwidth}
  \centering
  \includegraphics[trim = 0mm 0mm 0mm 0mm, width=\textwidth]{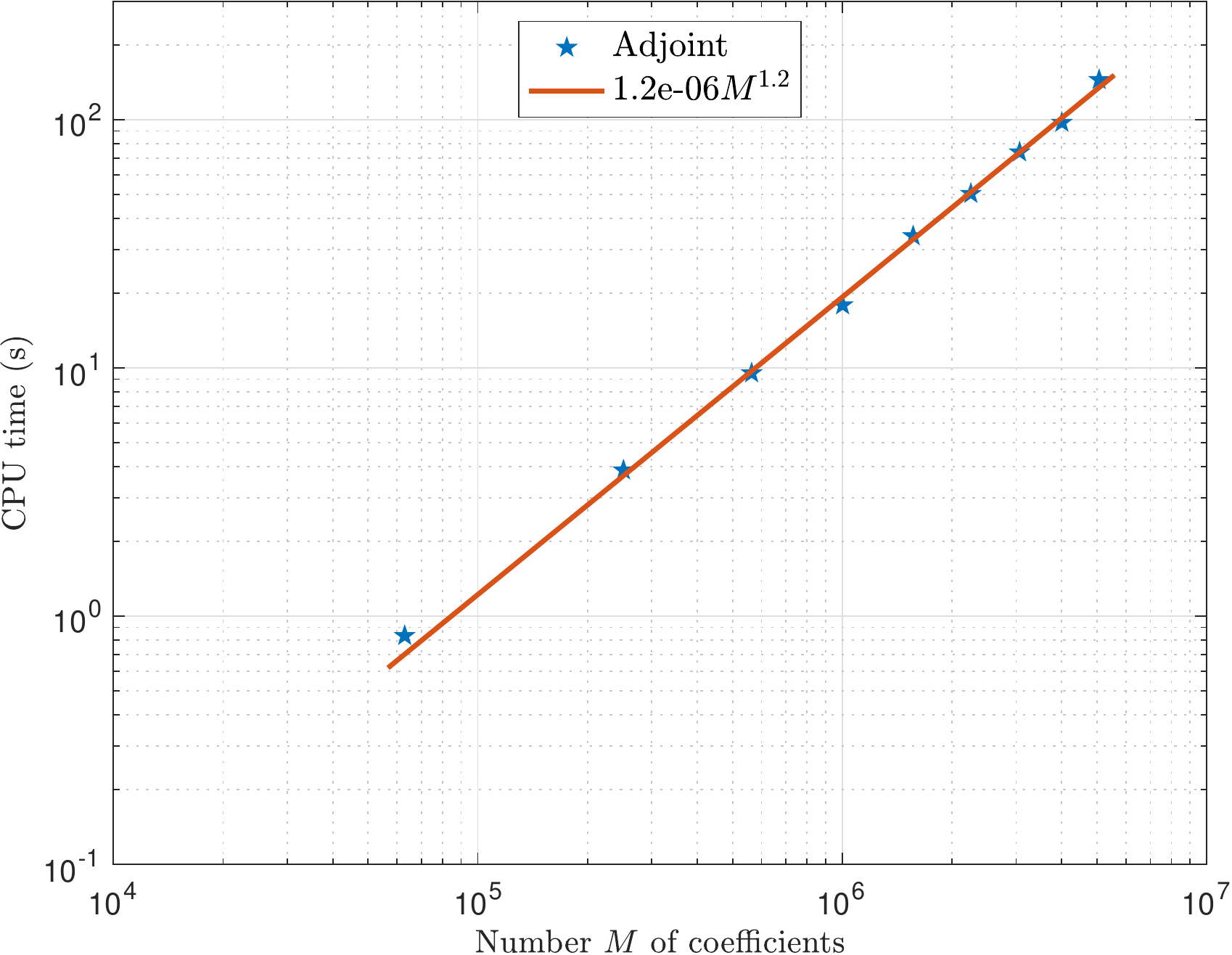}\\
  \subcaption{Adjoint {\fav}}\label{fig:time_faVest_adj}
  \end{minipage}
\end{minipage}
\begin{minipage}{0.9\textwidth}
\vspace{-1mm}
\caption{CPU time of Forward and Adjoint {\fav}s.}
\label{fig:timecost}
\end{minipage}
\end{minipage}
\end{figure}

\subsection{Numerical Stability}
We test the numerical stability of {\fav}. The stability can be measured by the max ratio of the discrete Fourier transform \eqref{eq:fwdvsht2} and the sum of absolute values of individual terms in \eqref{eq:almblmviaFlm}. See for example \cite{higham2002accuracy}. For example, for forward {\fav}, the ratio is simplified as 
\begin{equation}\label{eq:r_stable}
	\widehat{r}_{L,N} = \max_{\ell\leq L,|m|=1,\dots,\ell
	\atop k=1,\dots,N}\frac{|\dsh^{*}(\bx_k)|}{\widehat{V}_{\ell,m,k}}
	\qquad \widetilde{r}_{L,N} = \max_{\ell\leq L,|m|=1,\dots,\ell
	\atop k=1,\dots,N}\frac{|\dsh^{*}(\bx_k)|}{\widetilde{V}_{\ell,m,k}},
\end{equation}
where 
\begin{align*}
	\widehat{V}_{\ell,m,k} &= \sqrt{2}\Bigl(
	\xi_{\ell-1,m-1}^{(1)}\bigl|\shY[\ell-1,m-1]^*(\bx_k)\bigr|+\xi_{\ell+1,m-1}^{(2)}\bigl|\shY[\ell+1,m-1]^*(\bx_k)\bigr|
	+\xi_{\ell-1,m+1}^{(3)}\bigl|\shY[\ell-1,m+1]^*(\bx_k)\bigr|+\xi_{\ell+1,m+1}^{(4)}\bigl|\shY[\ell+1,m+1]^*(\bx_k)\bigr|\\
	& \quad +\xi_{\ell-1,m}^{(5)}\bigl|\shY[\ell-1,m]^*(\bx_k)\bigr|+\xi_{\ell+1,m}^{(6)}\bigl|\shY[\ell+1,m]^*(\bx_k)\bigr|
	\Bigr)\\
	\widetilde{V}_{\ell,m,k} &= \sqrt{2}\Bigl(\mu_{\ell,m-1}^{(1)}\bigl|\shY[\ell,m-1]^*(\bx_k)\bigr|
	+\mu_{\ell,m+1}^{(3)}\bigl|\shY[\ell,m+1]^*(\bx_k)\bigr|
	+\mu_{\ell,m}^{(2)}\bigl|\shY[\ell,m]^*(\bx_k)\bigr|
	\Bigr).
\end{align*}
The top row of Figure~\ref{fig:stability} shows the $\widehat{r}_{5,N}$ and $\widetilde{r}_{5,N}$ for $N$ from 100 up to 10,000. The errors in two pictures in the top panel both have an envelope with an increasing trend as the number of points $N$ increases.
The bottom row of Figure~\ref{fig:stability} shows the $\widehat{r}_{5,N}/N$ and $\widetilde{r}_{5,N}/N$ for $N$ from 100 up to 10,000. The errors in two pictures in the bottom panel both have an envelope with a decreasing trend as $N$ increases.
Similar trending results with respect to $N$ are illustrated in Figure~\ref{fig:stability L10}, where the top row shows the $\widehat{r}_{10,N}$ and $\widetilde{r}_{10,N}$, and the bottom row shows the the $\widehat{r}_{10,N}/N$ and $\widetilde{r}_{10,N}/N$, for $N$ from 100 up to 100,000.  
They illustrate that forward {\fav} has good stability. Computing similar quantities as \eqref{eq:r_stable} we can show the adjoint {\fav} also has good stability.

\begin{figure}[t]
\centering
\begin{minipage}{\textwidth}
\centering	
	\begin{minipage}{0.45\textwidth}
	\centering
		\includegraphics[width=\textwidth]{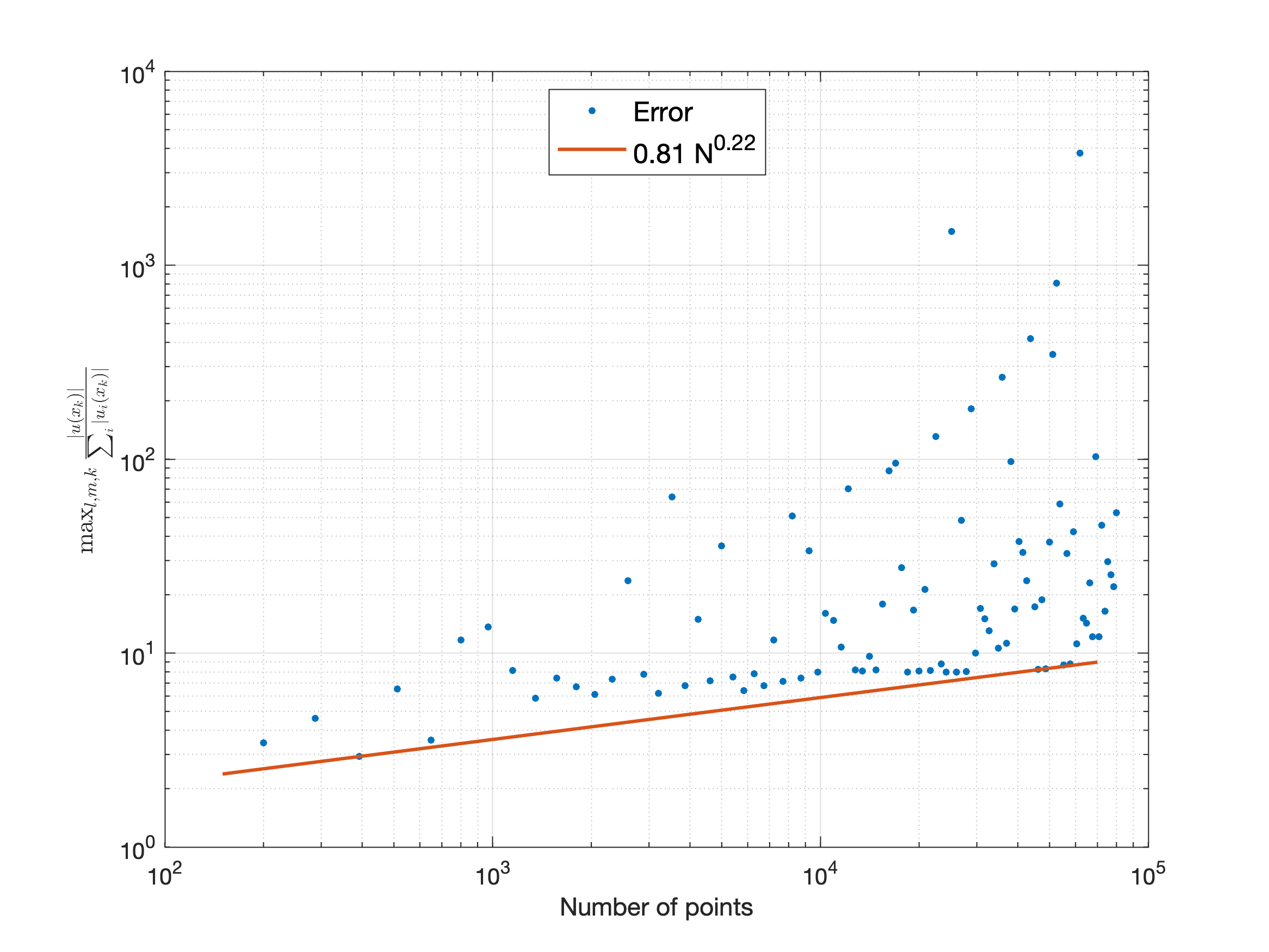}
		\subcaption{$\widehat{r}_{5,N}$}
	\end{minipage}
	\begin{minipage}{0.45\textwidth}
	\centering
		\includegraphics[width=\textwidth]{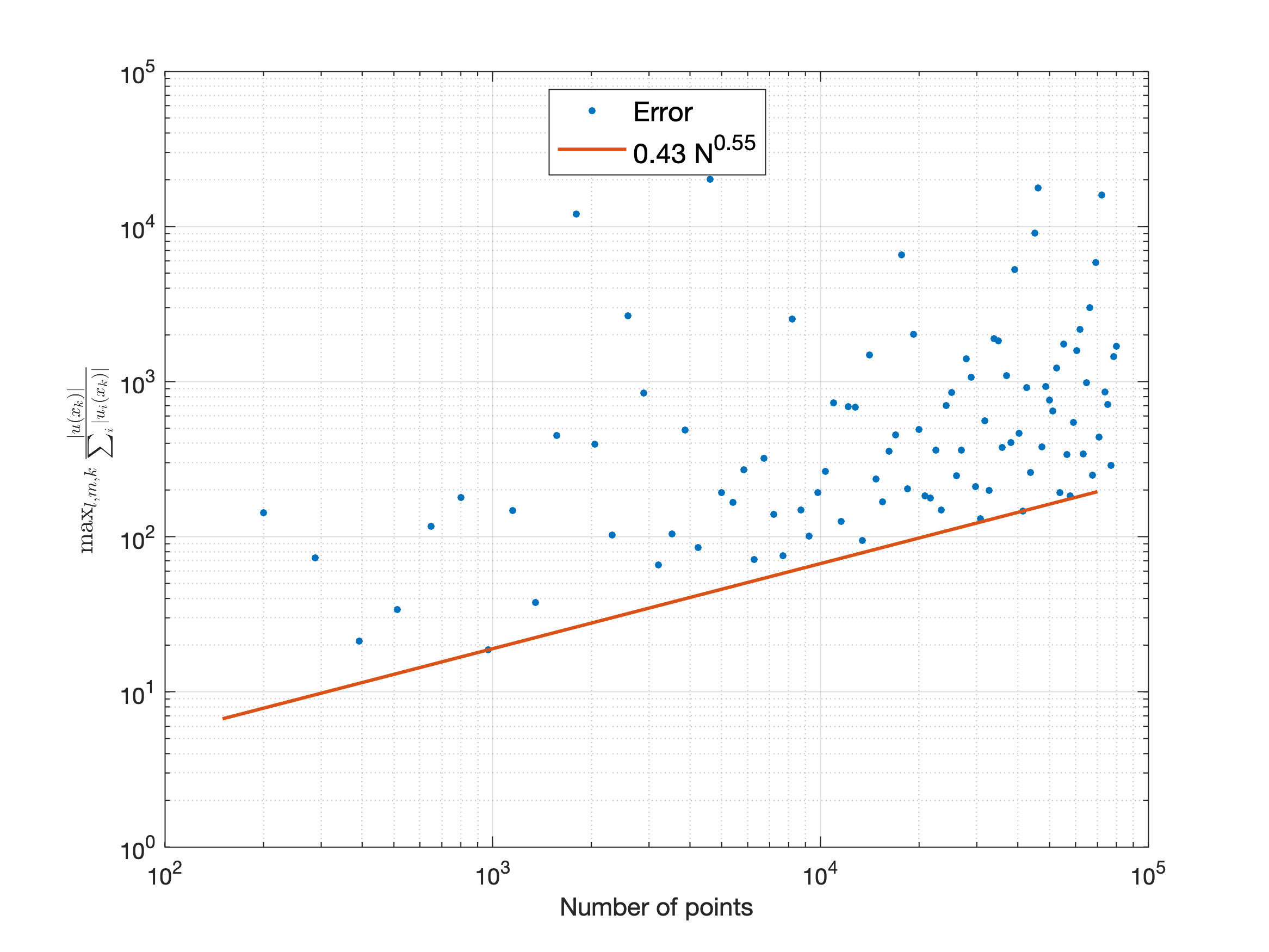}
		\subcaption{$\widetilde{r}_{5,N}$}
	\end{minipage}
\end{minipage}
\begin{minipage}{\textwidth}
\centering	
	\begin{minipage}{0.45\textwidth}
	\centering
		\includegraphics[width=\textwidth]{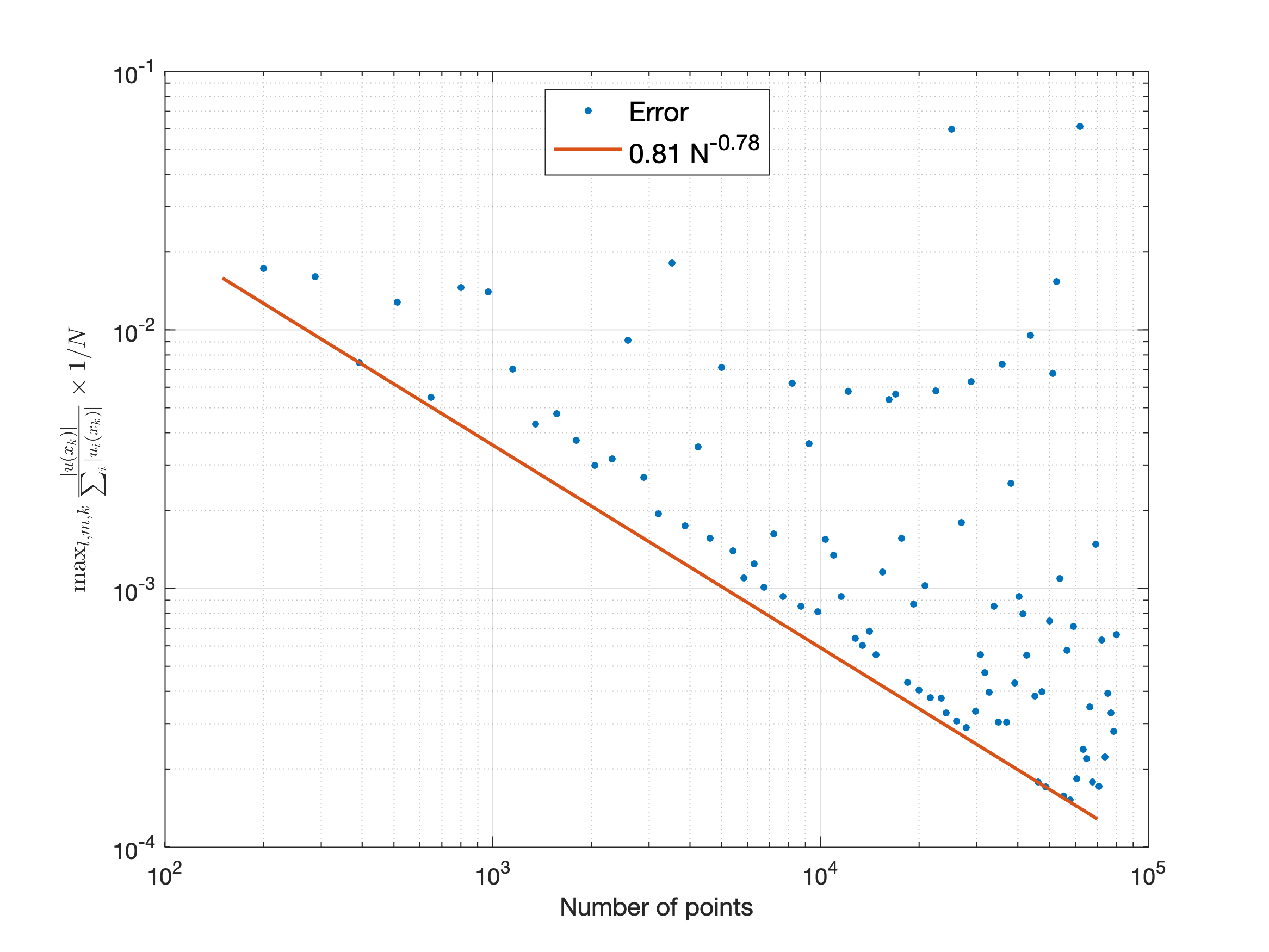}
		\subcaption{$\widehat{r}_{5,N}/N$}
	\end{minipage}
	\begin{minipage}{0.45\textwidth}
	\centering
		\includegraphics[width=\textwidth]{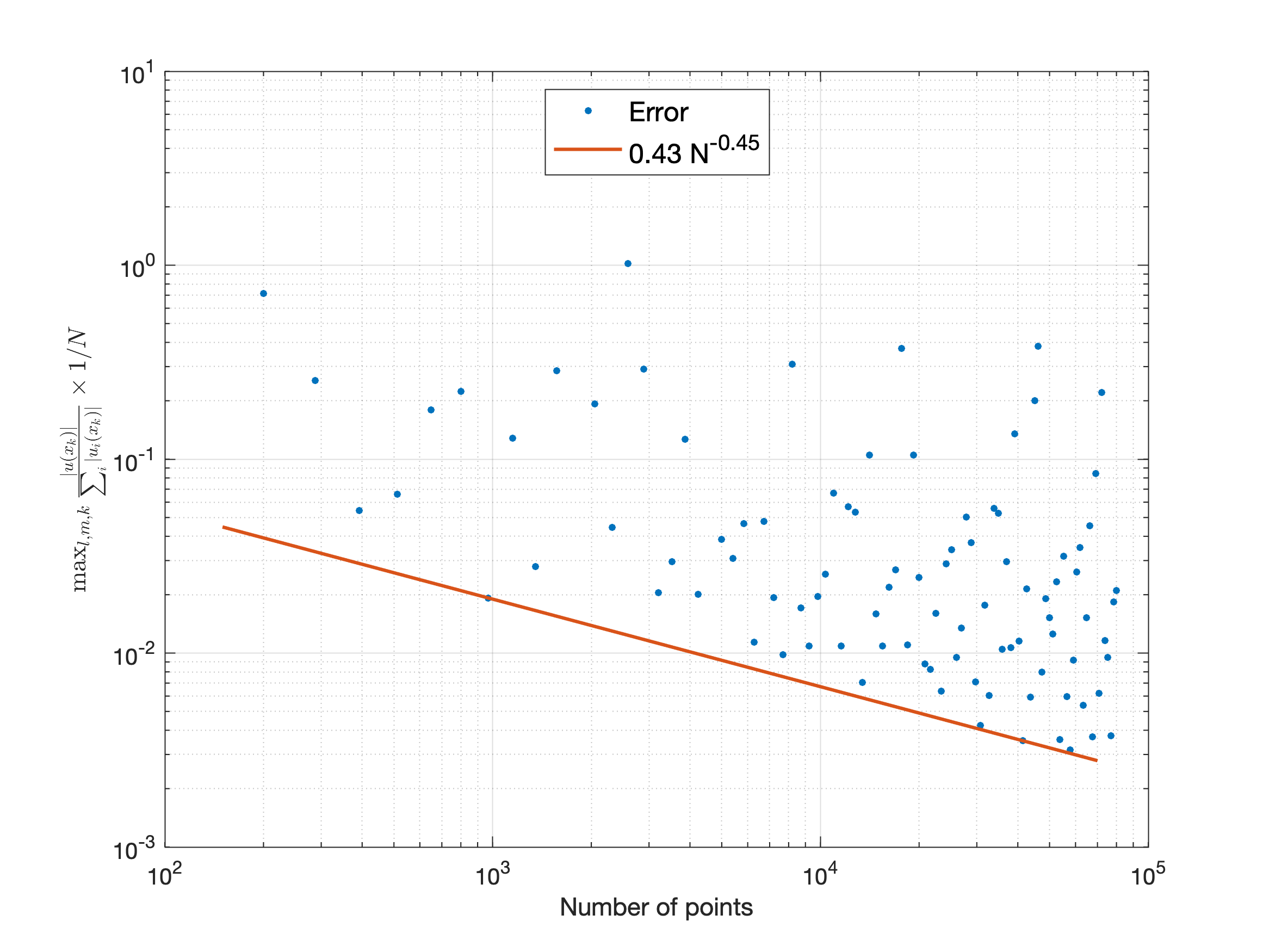}
		\subcaption{$\widetilde{r}_{5,N}/N$}
	\end{minipage}
\end{minipage}
\begin{minipage}{0.9\textwidth}
\caption{Stability test for forward {\fav}. The top row shows that the quantities $\widehat{r}_{5,N}$ and $\widetilde{r}_{5,N}$ have an envelope with positive power rate, which indicates their increasing trend as number of points $N$ increases. The bottom row shows that the quantities $\widehat{r}_{5,N}/N$ and $\widetilde{r}_{5,N}/N$ both have an envelope with negative power rate, which indicates their decreasing trend as $N$ increases.}\label{fig:stability}
\end{minipage}
\end{figure}

\begin{figure}[t]
\centering
\begin{minipage}{\textwidth}
\centering	
	\begin{minipage}{0.45\textwidth}
	\centering
		\includegraphics[width=\textwidth]{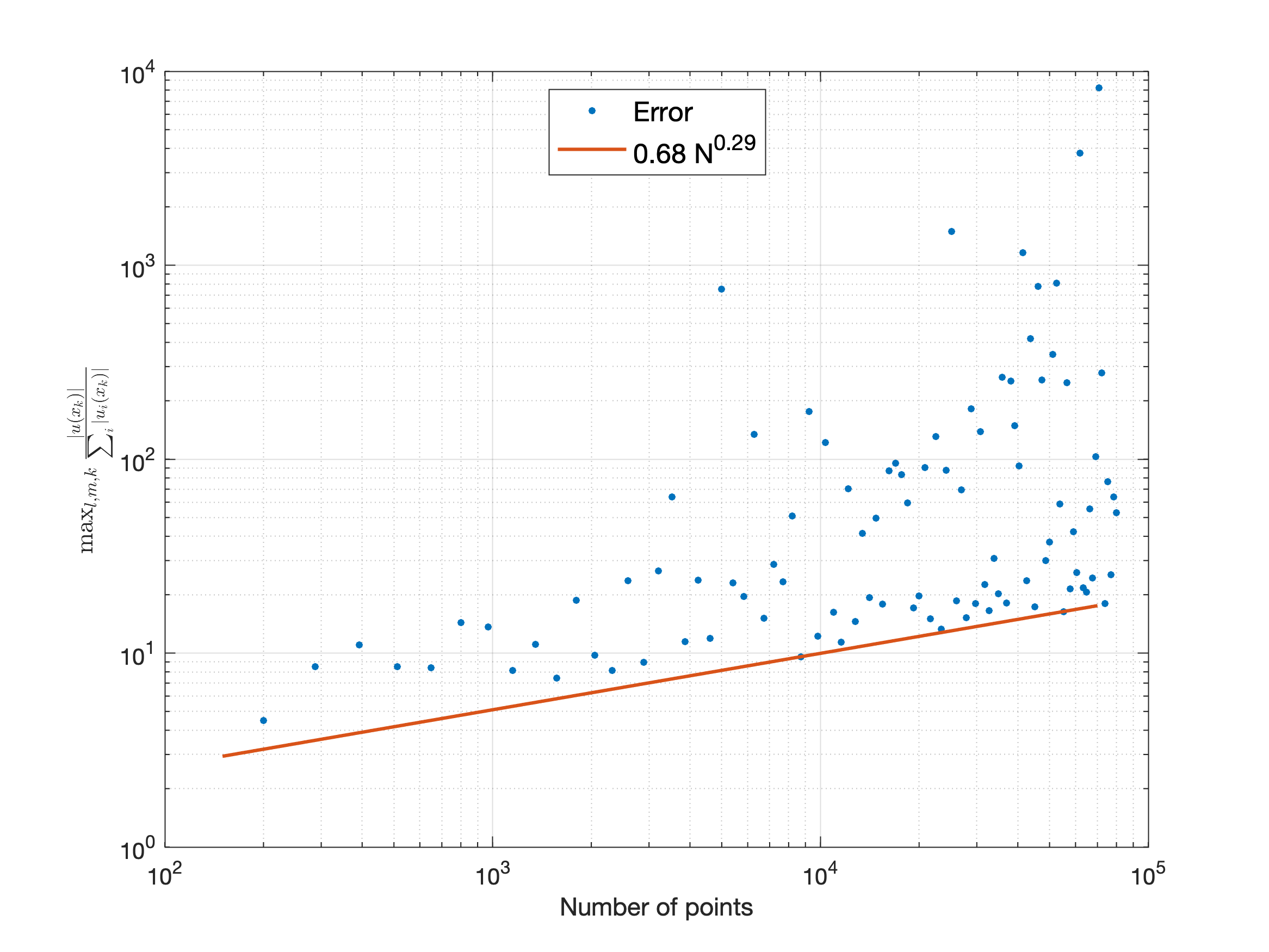}
		\subcaption{$\widehat{r}_{10,N}$}
	\end{minipage}
	\begin{minipage}{0.45\textwidth}
	\centering
		\includegraphics[width=\textwidth]{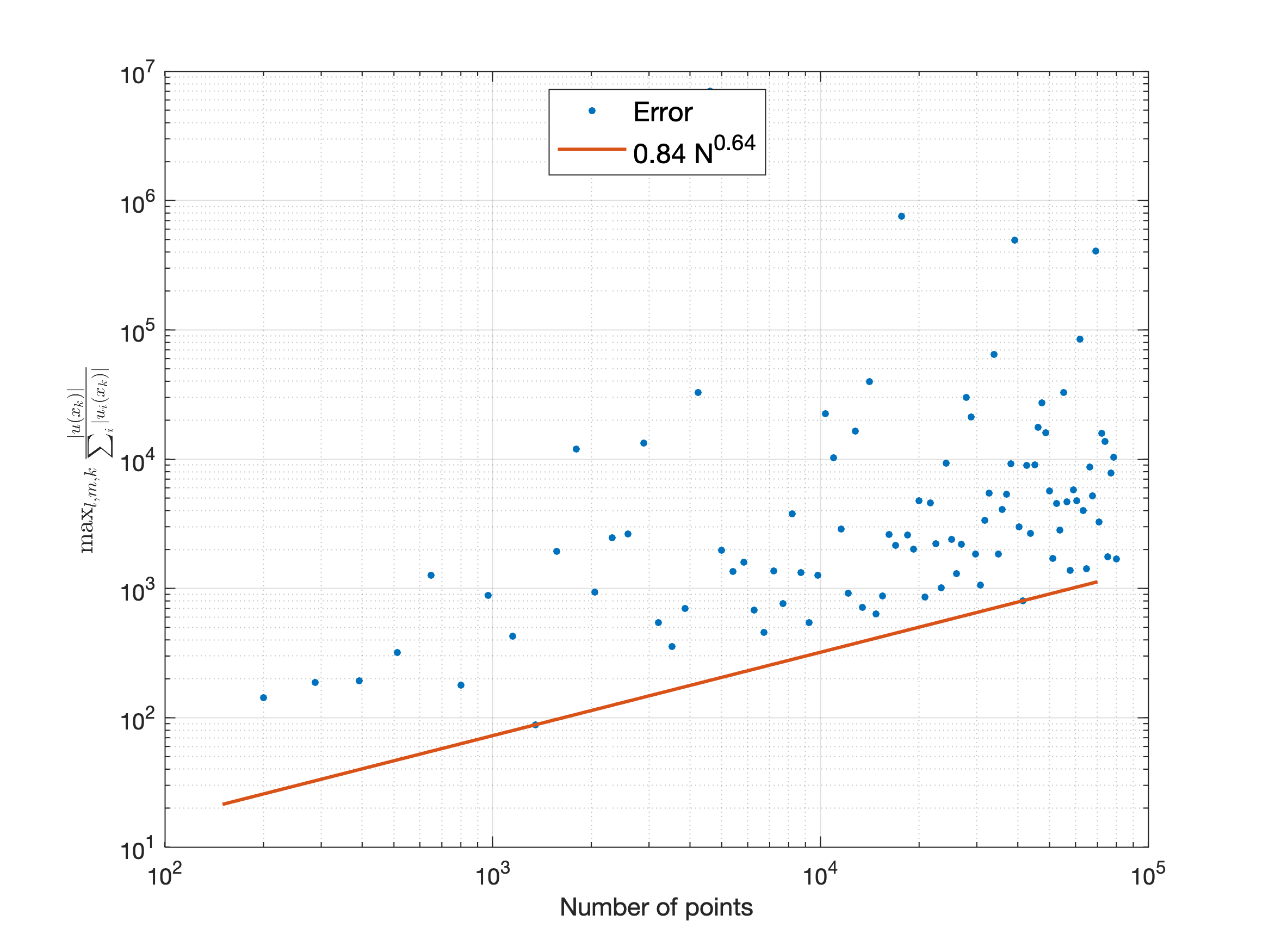}
		\subcaption{$\widetilde{r}_{10,N}$}
	\end{minipage}
\end{minipage}
\begin{minipage}{\textwidth}
\centering	
	\begin{minipage}{0.45\textwidth}
	\centering
		\includegraphics[width=\textwidth]{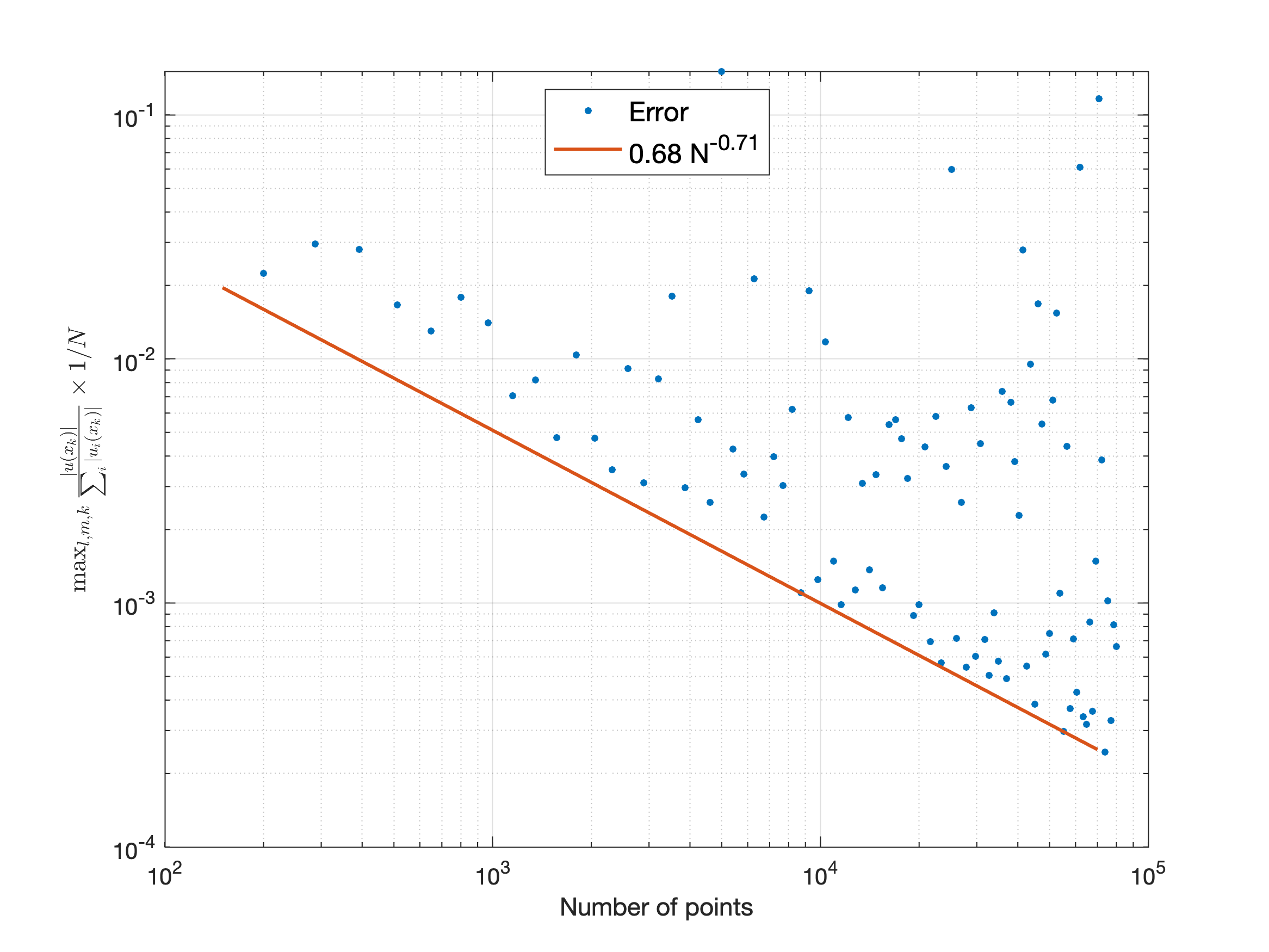}
		\subcaption{$\widehat{r}_{10,N}/N$}
	\end{minipage}
	\begin{minipage}{0.45\textwidth}
	\centering
		\includegraphics[width=\textwidth]{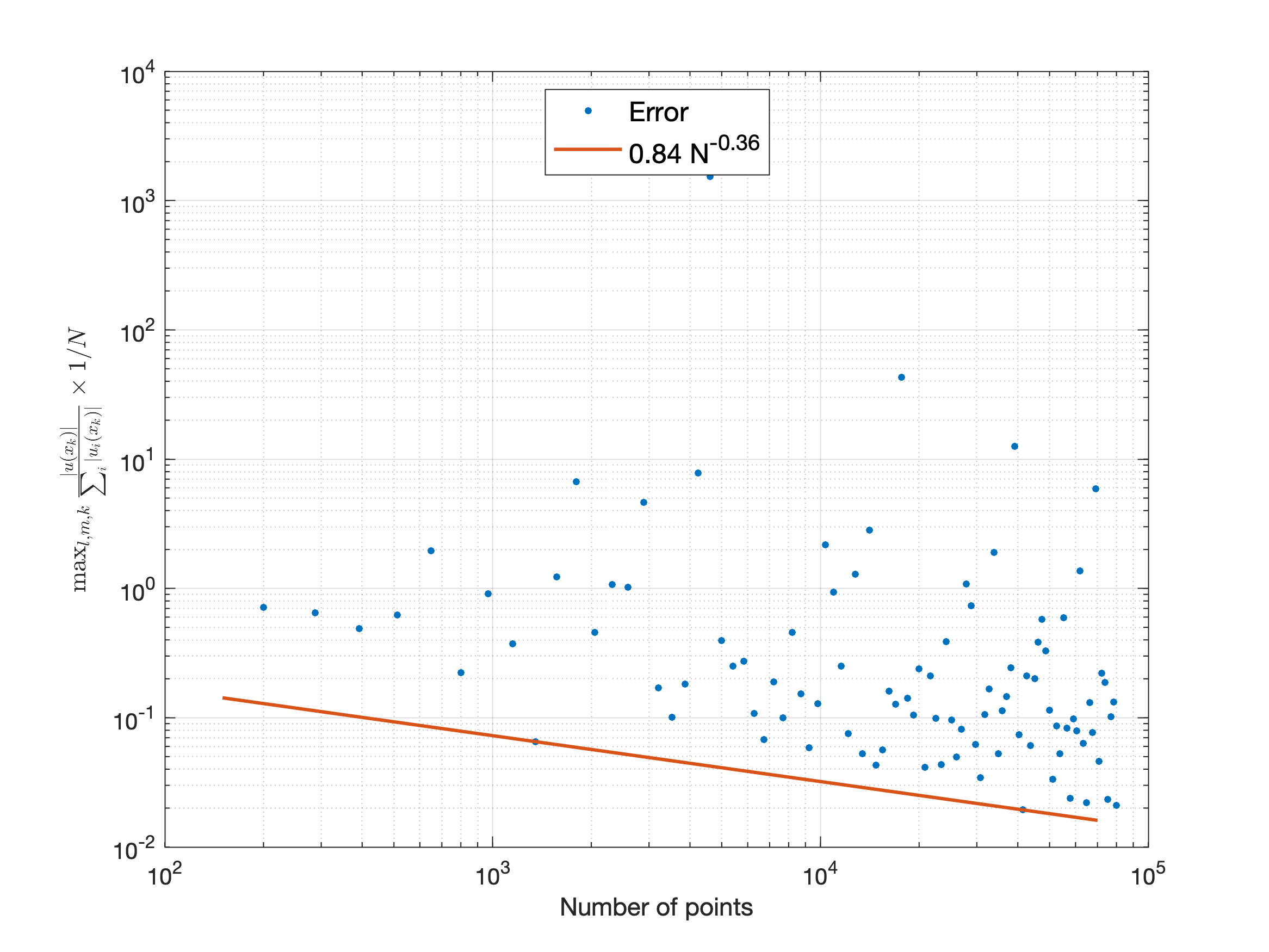}
		\subcaption{$\widetilde{r}_{10,N}/N$}
	\end{minipage}
\end{minipage}
\begin{minipage}{0.9\textwidth}
\caption{Stability test for forward {\fav}. The top row shows that the quantities $\widehat{r}_{10,N}$ and $\widetilde{r}_{10,N}$ have an envelope with positive power rate, which indicates their increasing trend as number of points $N$ increases. The bottom row shows that the quantities $\widehat{r}_{10,N}/N$ and $\widetilde{r}_{10,N}/N$ both have an envelope with negative power rate, which indicates their decreasing trend as $N$ increases.}\label{fig:stability L10}
\end{minipage}
\end{figure}

\section{Conclusions and Discussion}
This work proposes the first concrete fast algorithm which evaluates the forward and adjoint transforms of vector spherical harmonics for tangent fields. The fast algorithm (which we call {\fav}) is made possible from the representation of FwdVSHT and AdjVSHT by scalar spherical harmonics and Clebsch-Gordan coefficients.
By scalar FFTs on the sphere, the proposed algorithms for FwdVSHT and AdjVSHT both achieve the near-linear computational complexity.
The accuracy, computational speed and numerical stability of {\fav} are validated by the numerical examples of simulated tangent fields. We develop a software package in Matlab for {\fav}, which works for polynomial-exact quadrature rules. This package has been used in the fast computation for tensor needlet transform --- a fast multiresolution analysis --- for tangent field on the sphere in \cite{LiBrOlWa2019}. In future work, we will use the package for solving partial differential equations on the sphere such as Stokes or Navier-Stokes equations on $\sph{2}$.
As full skype maps in cosmology are usually evaluated at HEALPix points \cite{Gorski_etal2005}, we will develop a software package of {\fav} for HEALPix in Python in future.

\begin{acks}
The authors thank the anonymous referees' helpful comments and E. J. Fuselier and G. B. Wright for providing their MATLAB program which generates simulated tangent fields. The authors also thank P. Broadbridge and A. Olenko for their helpful comments during the initial stage of this work. M. Li acknowledges support from the National Natural Science Foundation of China under Grant 61802132, and the Australian Research Council under Discovery Project DP160101366 when he worked with  P. Broadbridge and A. Olenko at La Trobe University. Q. T. Le Gia and Y. G. Wang acknowledges support from the Australian Research Council under Discovery Project DP180100506. 
\end{acks}
\bibliographystyle{abbrv}
\bibliography{favest}

\begin{thebibliography}{10}

\bibitem{Planck2015Gravitational}
P.~A. Ade, N.~Aghanim, M.~Arnaud, M.~Ashdown, J.~Aumont, C.~Baccigalupi,
  A.~Banday, R.~Barreiro, J.~Bartlett, N.~Bartolo, et~al.
\newblock {Planck 2015 results-XV. Gravitational lensing}.
\newblock {\em Astronomy \& Astrophysics}, 594:A15, 2016.

\bibitem{Planck2018I}
N.~Aghanim, Y.~Akrami, F.~Arroja, M.~Ashdown, J.~Aumont, C.~Baccigalupi,
  M.~Ballardini, A.~J. Banday, R.~Barreiro, N.~Bartolo, et~al.
\newblock {Planck 2018 results-I. Overview and the cosmological legacy of
  Planck}.
\newblock {\em Astronomy \& Astrophysics}, 641:A1, 2020.

\bibitem{Planck2018Gravitational}
N.~Aghanim, Y.~Akrami, M.~Ashdown, J.~Aumont, C.~Baccigalupi, M.~Ballardini,
  A.~J. Banday, R.~Barreiro, N.~Bartolo, S.~Basak, et~al.
\newblock {Planck 2018 results-VIII. Gravitational lensing}.
\newblock {\em Astronomy \& Astrophysics}, 641:A8, 2020.

\bibitem{BaMaKl2018}
A.~H. Barnett, J.~Magland, and L.~af~Klinteberg.
\newblock A parallel nonuniform fast fourier transform library based on an
  ``exponential of semicircle'' kernel.
\newblock {\em SIAM Journal on Scientific Computing}, 41(5):C479--C504, 2019.

\bibitem{BaEsGi1985}
R.~G. Barrera, G.~A. Estevez, and J.~Giraldo.
\newblock Vector spherical harmonics and their application to magnetostatics.
\newblock {\em European Journal of Physics}, 6(4):287--294, oct 1985.

\bibitem{Brauchart_etal2015}
J.~Brauchart, J.~Dick, E.~Saff, I.~Sloan, Y.~Wang, and R.~Womersley.
\newblock {Covering of spheres by spherical caps and worst-case error for equal
  weight cubature in Sobolev spaces}.
\newblock {\em Journal of Mathematical Analysis and Applications},
  431(2):782--811, 2015.

\bibitem{BrSaSlWo2014}
J.~S. Brauchart, E.~B. Saff, I.~H. Sloan, and R.~S. Womersley.
\newblock Q{MC} designs: optimal order quasi {M}onte {C}arlo integration
  schemes on the sphere.
\newblock {\em Mathematics of Computation}, 83(290):2821--2851, 2014.

\bibitem{BrighamFFT}
E.~O. Brigham and E.~O. Brigham.
\newblock {\em The Fast Fourier Transform and its Applications}, volume 448.
\newblock Prentice Hall Englewood Cliffs, NJ, 1988.

\bibitem{DaXu2013}
F.~Dai and Y.~Xu.
\newblock {\em Approximation Theory and Harmonic Analysis on Spheres and
  Balls}.
\newblock Springer Monographs in Mathematics. Springer, New York, 2013.

\bibitem{NIST:DLMF}
{NIST Digital Library of Mathematical Functions}.
\newblock \url{http://dlmf.nist.gov}, Release 1.0.9 of 2014-08-29, 2014.
\newblock Online companion to \cite{Olver:2010:NHMF}.

\bibitem{drake2008algorithm}
J.~B. Drake, P.~Worley, and E.~D${}^\prime$Azevedo.
\newblock Algorithm 888: Spherical harmonic transform algorithms.
\newblock {\em ACM Transactions on Mathematical Software}, 35(3):23, 2008.

\bibitem{Edmonds2016}
A.~R. Edmonds.
\newblock {\em Angular Momentum in Quantum Mechanics}.
\newblock Princeton University Press, 2016.

\bibitem{Fan_etal2018}
M.~Fan, D.~Paul, T.~C.~M. Lee, and T.~Matsuo.
\newblock Modeling tangential vector fields on a sphere.
\newblock {\em Journal of the American Statistical Association},
  113(524):1625--1636, 2018.

\bibitem{FrGeSc1998}
W.~Freeden, T.~Gervens, and M.~Schreiner.
\newblock {\em Constructive Approximation on the Sphere}.
\newblock Numerical Mathematics and Scientific Computation. The Clarendon
  Press, Oxford University Press, New York, 1998.

\bibitem{Freeden2009}
W.~Freeden and M.~Schreiner.
\newblock {\em Spherical Functions of Mathematical Geosciences: A Scalar,
  Vectorial, and Tensorial Setup}.
\newblock Springer-Verlag, 2009.

\bibitem{FuWr2009}
E.~J. Fuselier and G.~B. Wright.
\newblock {Stability and error estimates for vector field interpolation and
  decomposition on the sphere with RBFs}.
\newblock {\em SIAM Journal on Numerical Analysis}, 47(5):3213--3239, 2009.

\bibitem{GaLeSl2011}
M.~Ganesh, Q.~T. Le~Gia, and I.~H. Sloan.
\newblock A pseudospectral quadrature method for {N}avier-{S}tokes equations on
  rotating spheres.
\newblock {\em Mathematics of Computation}, 80(275):1397--1430, 2011.

\bibitem{LeSlWaWo2017}
Q.~T.~L. Gia, I.~H. Sloan, Y.~G. Wang, and R.~S. Womersley.
\newblock Needlet approximation for isotropic random fields on the sphere.
\newblock {\em Journal of Approximation Theory}, 216:86--116, 2017.

\bibitem{Giraldo2004}
F.~Giraldo and T.~Rosmond.
\newblock {A scalable spectral element Eulerian atmospheric model (SEE-AM) for
  NWP: Dynamical core tests}.
\newblock {\em Monthly Weather Review}, 132:133--153, 2004.

\bibitem{Gorski_etal2005}
K.~M. Gorski, E.~Hivon, A.~J. Banday, B.~D. Wandelt, F.~K. Hansen, M.~Reinecke,
  and M.~Bartelmann.
\newblock {HEALPix: A framework for high-resolution discretization and fast
  analysis of data distributed on the sphere}.
\newblock {\em The Astrophysical Journal}, 622(2):759--771, apr 2005.

\bibitem{graf2013efficient}
D.-M.~M. Gr\"{a}f.
\newblock {\em Efficient algorithms for the computation of optimal quadrature
  points on Riemannian manifolds}.
\newblock PhD thesis, Universit\"{a}tsverlag Chemnitz, 2013.

\bibitem{graf2011computation}
M.~Gr\"{a}f and D.~Potts.
\newblock On the computation of spherical designs by a new optimization
  approach based on fast spherical fourier transforms.
\newblock {\em Numerische Mathematik}, 119(4):699--724, 2011.

\bibitem{Grafarend1986}
E.~W. Grafarend.
\newblock Three-dimensional deformation analysis: Global vector spherical
  harmonic and local finite element representation.
\newblock {\em Tectonophysics}, 130(1):337--359, 1986.

\bibitem{healy2004towards}
D.~Healy, P.~J. Kostelec, and D.~Rockmore.
\newblock {Towards safe and effective high-order Legendre transforms with
  applications to FFTs for the 2-sphere}.
\newblock {\em Advances in Computational Mathematics}, 21(1-2):59--105, 2004.

\bibitem{healy2003ffts}
D.~M. Healy, D.~N. Rockmore, P.~J. Kostelec, and S.~Moore.
\newblock {FFTs for the 2-sphere-improvements and variations}.
\newblock {\em Journal of Fourier Analysis and Applications}, 9(4):341--385,
  2003.

\bibitem{HeSl2005Optimal}
K.~Hesse and I.~H. Sloan.
\newblock Optimal lower bounds for cubature error on the sphere {$S^2$}.
\newblock {\em Journal of Complexity}, 21(6):790--803, 2005.

\bibitem{HeSl2005Worst}
K.~Hesse and I.~H. Sloan.
\newblock Worst-case errors in a {S}obolev space setting for cubature over the
  sphere {$S^2$}.
\newblock {\em Bulletin of the Australian Mathematical Society}, 71(1):81--105,
  2005.

\bibitem{HeSl2006Cubature}
K.~Hesse and I.~H. Sloan.
\newblock Cubature over the sphere {$S^2$} in {S}obolev spaces of arbitrary
  order.
\newblock {\em Journal of Approximation Theory}, 141(2):118--133, 2006.

\bibitem{HeSlWo2015}
K.~Hesse, I.~H. Sloan, and R.~S. Womersley.
\newblock {\em Numerical Integration on the Sphere}, pages 2671--2710.
\newblock Springer Berlin Heidelberg, Berlin, Heidelberg, 2015.

\bibitem{HeWo2012}
K.~Hesse and R.~S. Womersley.
\newblock Numerical integration with polynomial exactness over a spherical cap.
\newblock {\em Advances in Computational Mathematics}, 36(3):451--483, 2012.

\bibitem{higham2002accuracy}
N.~J. Higham.
\newblock {\em Accuracy and Stability of Numerical Algorithms}.
\newblock SIAM, 2002.

\bibitem{Hill1954}
E.~L. Hill.
\newblock The theory of vector spherical harmonics.
\newblock {\em American Journal of Physics}, 22(4):211--214, 1954.

\bibitem{Holton1973}
J.~R. Holton.
\newblock An introduction to dynamic meteorology.
\newblock {\em American Journal of Physics}, 41(5):752--754, 1973.

\bibitem{KeKuPo2020}
J.~Keiner, S.~Kunis, and D.~Potts.
\newblock {NFFT 3.5, C subroutine library}.
\newblock Contributors: F. Bartel, M. Fenn, T. Gorner, M. Kircheis, T. Knopp,
  M. Quellmalz, M. Schmischke, T. Volkmer, A. Vollrath.

\bibitem{KeKuPo2007}
J.~Keiner, S.~Kunis, and D.~Potts.
\newblock Efficient reconstruction of functions on the sphere from scattered
  data.
\newblock {\em Journal of Fourier Analysis and Applications}, 13(4):435--458,
  Aug 2007.

\bibitem{KeKuPo2009}
J.~Keiner, S.~Kunis, and D.~Potts.
\newblock {Using NFFT 3 --- a software library for various nonequispaced fast
  Fourier transforms}.
\newblock {\em ACM Transactions on Mathematical Software}, 36(4):19, 2009.

\bibitem{keiner2008fast}
J.~Keiner and D.~Potts.
\newblock Fast evaluation of quadrature formulae on the sphere.
\newblock {\em Mathematics of Computation}, 77(261):397--419, 2008.

\bibitem{Kim2009}
S.~Kim.
\newblock {The complete superconformal index for $N=6$ Chern-Simons theory}.
\newblock {\em Nuclear Physics B}, 821(1):241--284, 2009.

\bibitem{kunis2003fast}
S.~Kunis and D.~Potts.
\newblock {Fast spherical Fourier algorithms}.
\newblock {\em Journal of Computational and Applied Mathematics},
  161(1):75--98, 2003.

\bibitem{LeMh2008}
Q.~T. Le~Gia and H.~N. Mhaskar.
\newblock Localized linear polynomial operators and quadrature formulas on the
  sphere.
\newblock {\em SIAM Journal on Numerical Analysis}, 47(1):440--466, 2008/09.

\bibitem{LiBrOlWa2019}
M.~Li, P.~Broadbridge, A.~Olenko, and Y.~G. Wang.
\newblock Fast tensor needlet transforms for tangent vector fields on the
  sphere.
\newblock {\em arXiv preprint arXiv:1907.13339}, 2019.

\bibitem{Lowes1966}
F.~J. Lowes.
\newblock Mean-square values on sphere of spherical harmonic vector fields.
\newblock {\em Journal of Geophysical Research}, 71(8):2179--2179, 1966.

\bibitem{lu2013efficient}
F.-s. Lu, J.-q. Song, W.-q. Lin, Y.-f. Pang, K.-j. Ren, and P.-c. Shi.
\newblock {Efficient utilization of launched threads on GPUs: The spherical
  harmonic transform as a case study}.
\newblock {\em Computer Physics Communications}, 184(11):2494--2502, 2013.

\bibitem{MaPe2011}
D.~Marinucci and G.~Peccati.
\newblock {\em Random Fields on the Sphere}, volume 389 of {\em London
  Mathematical Society Lecture Note Series}.
\newblock Cambridge University Press, Cambridge, 2011.

\bibitem{MhNaWa2001}
H.~N. Mhaskar, F.~J. Narcowich, and J.~D. Ward.
\newblock Spherical {M}arcinkiewicz-{Z}ygmund inequalities and positive
  quadrature.
\newblock {\em Mathematics of Computation}, 70(235):1113--1130, 2001.

\bibitem{mohlenkamp1999fast}
M.~J. Mohlenkamp.
\newblock A fast transform for spherical harmonics.
\newblock {\em Journal of Fourier Analysis and Applications}, 5(2-3):159--184,
  1999.

\bibitem{Nedelec2001}
J.~C. N'{e}d\'{e}lec.
\newblock {\em Acoustic and Electromagnetic Equations}, volume 144 of {\em
  Applied Mathematical Sciences}.
\newblock Springer-Verlag, New York, 2001.

\bibitem{pekurovsky2012p3dfft}
D.~Pekurovsky.
\newblock {P3DFFT: A framework for parallel computations of Fourier transforms
  in three dimensions}.
\newblock {\em SIAM Journal on Scientific Computing}, 34(4):C192--C209, 2012.

\bibitem{press2007numerical}
W.~H. Press, S.~A. Teukolsky, W.~T. Vetterling, and B.~P. Flannery.
\newblock {\em Numerical Recipes 3rd Edition: The Art of Scientific Computing}.
\newblock Cambridge University Press, 2007.

\bibitem{ReSe2013}
{Reinecke, M.} and {Seljebotn, D. S.}
\newblock Libsharp --- spherical harmonic transforms revisited.
\newblock {\em Astronomy \& Astrophysics}, 554:A112, 2013.

\bibitem{RoTy2006}
V.~Rokhlin and M.~Tygert.
\newblock Fast algorithms for spherical harmonic expansions.
\newblock {\em SIAM Journal on Scientific Computing}, 27(6):1903--1928, 2006.

\bibitem{SoWaWuYu2019}
M.~Sourisseau, Y.~G. Wang, H.-T. Wu, and W.-H. Yu.
\newblock Optimization-based quasi-uniform spherical t-design and generalized
  multitaper for complex physiological time series.
\newblock {\em arXiv preprint arXiv:1907.13493}, 2019.

\bibitem{suda2004stability}
R.~Suda.
\newblock {Stability analysis of the fast Legendre transform algorithm based on
  the fast multipole method}.
\newblock In {\em Proceedings of the Estonian Academy of Sciences}, volume~53,
  pages 107--115. Estonian Academy Publishers; 1999, 2004.

\bibitem{suda2005fast}
R.~Suda.
\newblock Fast spherical harmonic transform routine {FLTSS} applied to the
  shallow water test set.
\newblock {\em Monthly Weather Review}, 133(3):634--648, 2005.

\bibitem{suda2002fast}
R.~Suda and M.~Takami.
\newblock A fast spherical harmonics transform algorithm.
\newblock {\em Mathematics of Computation}, 71(238):703--715, 2002.

\bibitem{Swarztrauber1996}
P.~N. Swarztrauber.
\newblock Spectral transform methods for solving the shallow-water equations on
  the sphere.
\newblock {\em Monthly Weather Review}, 124(4):730--744, 1996.

\bibitem{Swarztrauber2004}
P.~N. Swarztrauber.
\newblock Shallow water flow on the sphere.
\newblock {\em Monthly Weather Review}, 132(12):3010--3018, 2004.

\bibitem{Swarztrauber2000}
P.~N. Swarztrauber and W.~F. Spotz.
\newblock Generalized discrete spherical harmonic transforms.
\newblock {\em Journal of Computational Physics}, 159(2):213--230, 2000.

\bibitem{tang2005dfti}
P.~T.~P. Tang.
\newblock {DFTI --- a new interface for Fast Fourier Transform libraries}.
\newblock {\em ACM Transactions on Mathematical Software}, 31(4):475--507,
  2005.

\bibitem{Tygert2008}
M.~Tygert.
\newblock {Fast algorithms for spherical harmonic expansions, II}.
\newblock {\em Journal of Computational Physics}, 227(8):4260--4279, 2008.

\bibitem{tygert2010fast}
M.~Tygert.
\newblock {Fast algorithms for spherical harmonic expansions, III}.
\newblock {\em Journal of Computational Physics}, 229(18):6181--6192, 2010.

\bibitem{Varshalovich_etal1988}
D.~A. Varshalovich, A.~N. Moskalev, and V.~K. Khersonskii.
\newblock {\em Quantum Theory of Angular Momentum}.
\newblock World Scientific, 1988.

\bibitem{VeGuBiZo2009}
S.~K. Veerapaneni, D.~Gueyffier, G.~Biros, and D.~Zorin.
\newblock {A numerical method for simulating the dynamics of 3D axisymmetric
  vesicles suspended in viscous flows}.
\newblock {\em Journal of Computational Physics}, 228(19):7233 -- 7249, 2009.

\bibitem{VeRaBiZo2011}
S.~K. Veerapaneni, A.~Rahimian, G.~Biros, and D.~Zorin.
\newblock A fast algorithm for simulating vesicle flows in three dimensions.
\newblock {\em Journal of Computational Physics}, 230(14):5610 -- 5634, 2011.

\bibitem{WaWaXi2018}
B.~Wang, L.-L. Wang, and Z.~Xie.
\newblock Accurate calculation of spherical and vector spherical harmonic
  expansions via spectral element grids.
\newblock {\em Advances in Computational Mathematics}, 44(3):951--985, Jun
  2018.

\bibitem{Wang2016}
Y.~Wang.
\newblock Filtered polynomial approximation on the sphere.
\newblock {\em Bulletin of the Australian Mathematical Society},
  93(1):162--163, 2016.

\bibitem{WaLeSlWo2017}
Y.~G. Wang, Q.~T.~L. Gia, I.~H. Sloan, and R.~S. Womersley.
\newblock Fully discrete needlet approximation on the sphere.
\newblock {\em Applied and Computational Harmonic Analysis}, 43(2):292--316,
  2017.

\bibitem{WaSlWo2018}
Y.~G. Wang, I.~H. Sloan, and R.~S. Womersley.
\newblock Riemann localisation on the sphere.
\newblock {\em Journal of Fourier Analysis and Applications}, 24(1):141--183,
  Feb 2018.

\bibitem{Weinberg1994}
E.~J. Weinberg.
\newblock Monopole vector spherical harmonics.
\newblock {\em Physical Review D}, 49:1086--1092, Jan 1994.

\bibitem{Williamson1992}
D.~L. Williamson, J.~B. Drake, J.~J. Hack, R.~Jakob, and P.~N. Swarztrauber.
\newblock A standard test set for numerical approximations to the shallow water
  equations in spherical geometry.
\newblock {\em Journal of Computational Physics}, 102(1):211--224, 1992.

\bibitem{Womersley_ssd_URL}
R.~S. Womersley.
\newblock Efficient spherical designs with good geometric properties.
\newblock In {\em Contemporary Computational Mathematics --- A Celebration of
  the 80th Birthday of Ian Sloan}, pages 1243--1285. Springer. URL:
  \url{https://web.maths.unsw.edu.au/~rsw/Sphere/EffSphDes/}, 2018.

\end{thebibliography}

\end{document}